\documentclass[preprint,12pt]{elsarticle}
\usepackage{amsfonts}
\usepackage{amsthm}
\usepackage{amsmath}
\usepackage[all]{xy}
\usepackage{titlesec}
\usepackage{dsfont}
\usepackage{amssymb}
\usepackage{extarrows}
\usepackage{mathrsfs}
\usepackage{arydshln}
\usepackage{multirow}
\usepackage{color}
\usepackage{pb-diagram}
\usepackage{times,epic,eepic,pb-diagram,tikz-cd}
\usepackage{graphicx}
\graphicspath{{img/}}
\usetikzlibrary{matrix,arrows,trees}
\usetikzlibrary{decorations.markings}
\usetikzlibrary{decorations.pathmorphing}

\usepackage[bookmarksnumbered, colorlinks, plainpages]{hyperref}
in \textwidth 6 true in

\journal{xx}
\newtheorem{theorem}{Theorem}[section]
\newtheorem{proposition}[theorem]{Proposition}
\newtheorem{lemma}[theorem]{Lemma}
\newtheorem{corollary}[theorem]{Corollary}

\newtheorem{remark}[theorem]{Remark}

\newcommand{\Z}{\ensuremath{\mathbb{Z}}}
\newcommand{\fhe}[0]{\ensuremath{{\scriptstyle\circ}}}
\newcommand{\Pn}{\ensuremath{P^{2k}(n)}}

\newcommand{\lr}[1]{\ensuremath{\left \langle #1  \right \rangle }}

\begin{document}

\begin{frontmatter}

\title{ Complexes equivalent to $S^{2k-1}$-fibrations over $S^{2k}$}

\author[1]{Zhongjian Zhu}
\ead{20160118@wzu.edu.cn}
\author[2]{Jianzhong Pan\corref{cor}}
\ead{pjz@amss.ac.cn}
\cortext[cor]{Corresponding author}
\address[1]{College of Mathematics and Physics, Wenzhou University, Wenzhou 325035, China}
\address[2]{Hua Loo-Keng Key Mathematical Laboratory, Institute of Mathematics,\\ Academy of Mathematics and Systems Science,
	Chinese Academy of Sciences; \\University of Chinese Academy of Sciences, Beijing,  100190, China}

\begin{abstract}
In this paper,  necessary and sufficient conditions are obtained for the attaching map $f$ of the top cell  of a CW complex  to have the homotopy type of the total space of $S^{2k-1}$-fibration over $S^{2k}$ for any $k\geq 2$. As an application, the order of  any attaching map of the top cell  of the total space of an $S^{2k-1}$-fibration over $S^{2k}$ is determined and when $k\ne 2,4$, the homotopy types of the total spaces of $S^{2k-1}$-fibrations over $S^{2k}$ are classified by the stable homotopy classes of the attaching maps. 
\end{abstract}

\end{frontmatter}

\section{Introduction}
\label{intro}

Sphere bundles over spheres are of great significance in geometric topology. 
In 1956, John Milnor \cite{Mil56} discovered the first example of so-called “exotic” spheres that were homeomorphic to
the 7-sphere but not diffeomorphic to it with its standard smooth structure.
Since Milnor's discovery of exotic spheres, there has been a sustained interest in whether exotic spheres admit Riemannian metrics with positive curvature.  Grove and Ziller  showed that the total spaces of those $S^3$-bundles over $S^4$ in Milnor's construction admit  metrics with non-negative sectional curvature \cite{Grove}

On the other hand, the unit tangent  sphere bundle is diffeomorphic to the Stiefel manifold, i.e., 
\[
STS^{2k} \cong V_{2k+1,2}.
\] 
It is an important tool for studying 2-frames on smooth manifolds \cite{Thomas}.

It is well known that the oriented preserving isomorphism classes of fiber bundles $X$ (We will not make a distinction between sphere bundles and their corresponding total spaces) with structure $SO(q+1)$ over \( S^r \) with fiber \( S^q \) (It will be abbreviated to sphere bundle in the subsequent text) correspond to the homotopy classes of maps \(\chi: S^{r-1} \to SO(q+1) \).  

At first glance, such a classification problem for isomorphisms receives a perfect homotopy-theoretic solution. However, some other related problems are not thereby resolved, e.g, the classification problem for the total spaces of these fiber bundles. Considering that non-isomorphic fiber bundles can have homeomorphic total spaces, the classification of total spaces is non-trivial. James and Whitehead initiated homotopy classification of sphere bundles over spheres \cite{JamesI,JamesII,James On sph bundle}. When $r-q>1$ or for sphere bundles with sections, the homotopy classification problems was relatively simple and solved by James and Whitehead \cite{JamesI,JamesII}. When $r=2k,q=2k-1,k\ne 2,4$ and the sphere bundles have no sections, a complete homotopy classification was obtained by James \cite{James On sph bundle}. 

To explain the difficulties when $k=2,4$, let's sketch briefly his argument.

For an $S^{2k-1}$-bundle $X$ over $S^{2k}$ without any section, the total space $X\simeq P^{2k}(n)\cup_f e^{4k-1}$ where $ P^{2k}(n)=S^{2k-1}\cup_n e^{2k}$ and $n$ is given by the Euler class of the bundle. 
If \(\chi: S^{2k-1} \to SO(2k) \) is the clutching function of the bundle $X$, then $J$-homomorphism acting on $\chi$ gives a map $\mu:S^{4k-1}\to S^{2k}$. The attaching map $f: S^{4k-2} \to  P^{2k}(n) $ and $\mu:S^{4k-1}\to S^{2k}$ are related by the following equation:
\begin{align}
	\Sigma f=i_{2k} \fhe \mu . \label{mu}
\end{align}
where $\Sigma $ is the suspension functor and $i_{2k}:S^{2k}\rightarrow \Pn$ is the canonical inclusion.

The homotopy class of $\mu$ was the invariant to classify, up to homotopy, the total space $X$.
To obtain such a homotopy classification, two results are essential for his argument to work.

The first is an identity 
\begin{align}
	j_\ast(f)=\pm [X_{2k},\iota_{2k-1}]_r   . \label{relative}
\end{align}
where $X_{2k}$ represents the characteristic map of the top cell of $ P^{2k}(n)$, $\iota_{2k-1}$ is the identity map over $S^{2k-1}$, $[ , ]_r$ is the relative Whitehead product and $j_\ast$ is induced by the inclusion $( P^{2k}(n),\ast)\hookrightarrow ( P^{2k}(n), S^{2k-1})$.

The second is the following result of homotopy nature: when $k\ne 2,4$, for an element $\xi \in \pi_{4k-2}(S^{2k-1})$ with $\Sigma (i_{2k-1}\fhe \xi)=0$, there exists another element $\xi' \in \pi_{4k-2}(S^{2k-1})$ such that $\Sigma \xi'=0$ and $i_{2k-1}\fhe \xi=i_{2k-1}\fhe \xi'$.
When $k=2,4$, this is impossible since the kernel of the homomorphism $i_{2k\ast}:\pi_{4k-1}(S^{2k})\to \pi_{4k-1}(P^{2k+1}(n))$ is too big.

By the homotopy classification of total spaces of  $S^{2k-1}$-fibrations over $S^{2k}$ obtained in this paper, we are able to complete the homotopy classification of total spaces of  $S^{2k-1}$-bundles $X$ over $S^{2k}$ using homotopy method. 
This is one of the motivations for our study on homotopy theory of $S^{2k-1}$-fibrations over $S^{2k}$. Before stating two other motivations, let's explain the history of this topic.

After James and Whitehead's work on homotopy of sphere-bundles over spheres, Sasao\cite{SASAO3} considered the homotopy classification of total spaces of  $S^{q}$-fibrations over $S^{r}$. When $r-q>1$ or for fibrations with sections, an identity like (\ref{relative}) was established by homotopy theory method, with the help of this identity, similar results to that of James and Whitehead were obtained. Sasao remarked that the case when $r-q=1$ was different and left open.

Our interest on homotopy of $S^{2k-1}$-fibrations over $S^{2k}$ comes from Theriault's work \cite{Theriault} on Homotopy exponent of  $ P^{2k}(2^r)$. In 2008, Theriault constructed  a special  $S^{2k-1}$-fibration over $S^{2k}$ in the category of 2-local spaces. 

Note that  for $R=\Z_2$ or $\Z_{2^r}$, $H_{\ast}(\Omega\Sigma P^{2k}(2^r);R)=T(u_{2k-1}, v_{2k})$ which is the freely generated $R$-tensor algebra with generators $u_{2k-1}\in H_{2k-1}(P^{2k}(2^r),R)$ and  $v_{2k}\in H_{2k}(P^{2k}(2^r),R)$. 
He considered the following homotopy fibration given by Cohen \cite[Lemma 21.1]{Cohen Course}
\begin{align}
	X_r\xrightarrow{\mathbf{i}_r} QP^{2k}(2^r):=\Omega^{\infty}\Sigma^{\infty} P^{2k}(2^r)\xrightarrow{f_r} K(\Z_2, 4k-2). \label{fib:X_r}
\end{align}
where $f_r$  is a representative of the cohomology class $u_{2k-1}^2\in H_{\ast}(Q P^{2k}(2^r),\Z_2)$. Then $\tau_r(S^{2k}):=(4k-1)$-skeleton of $X_r$ is a  total space of $S^{2k-1}$-fibration over $S^{2k}$ (under 2-localization), which has homotopy type $P^{2k}(2^r)\cup e^{4k-1}$. He called  $\tau_r(S^{2k})$ ``mod-$2^r$ tangent bundles"  since $\tau_1(S^{2k})$  is just the unit tangent bundle of $S^{2n}$. 

The last motivation is from Kitchloo and Shankar's work \cite{S3S4bundle} on $S^3$-bundles over $S^4$ which was motivated by the problem of classifying $S^3$-bundles over $S^4$  up to homotopy, homeomorphism, and diffeomorphism  raised by  Grove and Ziller  \cite{Grove}. In their paper,  Kitchloo and Shankar used a clever argument connecting the cup product on the total spaces and differentials in the Serre Spectral sequences for related fibrations related to the bundles.

We are able to generalize their spectral sequence argument to all $S^{2k-1}$-fibration over $S^{2k}$, connecting Hopf invariant of a homotopy version of $\mu$ and differentials in the Serre Spectral sequences for related fibrations related to the sphere fibrations. This is the first step to a complete homotopy classification of total spaces of  $S^{q}$-fibrations over $S^{r}$.

By the way, at about the same time,   D.Crowley and  C.Escher employed completely different methods to obtain the classification, up to homotopy, homeomorphism, and diffeomorphism, of the total spaces of 3-sphere bundles over the 4-sphere \cite{Crowley}.  

For  $S^{2k-1}$-fibration over $S^{2k}$, their total spaces are CW-complexes with three cells: 
\[
X = P^{2k}(n) \cup e^{4k-1}.
\] 
The integer $n$ in the aforementioned cellular CW-complex \( X \) depends only on the homotopy type of $X$. Therefore, determining the total spaces essentially reduces to the attaching map of the top-dimensional cell. Hence, determining the total space of such fibrations can be divided into the following two problems:

\begin{enumerate}
	\item [(i)] Determine the necessary and sufficient conditions for \( f : S^{4k-2} \rightarrow P^{2k}(n) \) to be the attaching map of the top-dimensional cell in an \( S^{2k-1} \)-fibration over \( S^{2k} \);
	\item [(ii)] Determine the homotopy types of the mapping cones $C_f$ of  \( f \) such that $C_f$ is a homotopy equivalent to the total space of  \( S^{2k-1} \)-fibration over \( S^{2k} \). 
\end{enumerate}

Kitchloo and Shankar solved the problem (i) for $2k=4$ and $8$ \cite{S3S4bundle}.

In this papar, we completely solve the Serre fibration version of the  problem (i)  for  all $k\geq 2$.  In a subsequent paper \cite{ZhuPanclassfyBundle}, we also completely resolve problem (ii) for small $k$.

To make sense of the introduction and state our main results, we need the following notations and conventions.  Unless explicitly stated otherwise, we will not distinguish between the notation for the homotopy or the equality  of two maps. 

For mod $n$ Moore space $P^{k+1}(n), n\geq 2$ of dimension $k+1$,  
there is a canonical homotopy cofibration 
\[S^{k} \xrightarrow{[n]} S^{k} \xrightarrow{i_k} P^{k+1}(n)\xrightarrow{p_{k+1}} S^{k+1},\] 
where  $\iota_{k}$ denotes the homotopy class of the identity map on sphere $S^k$ and  $[n]$ is the simplification of $n\iota_k$ if it will not cause confusion on the dimension of the sphere.

Suppose there is a cofibration sequence:
\begin{align}
	S^{4k-2} \overset{f}{\to} P^{2k}(n)\overset{i_X}{\to}  X, n\geq 2~\text{and}~ k\geq 2. \label{cofiber X}
\end{align}

Note that $\pi_{2k}(\Pn, S^{2k-1})=\Z\{X_{2k}\}$, where $X_{2k}$ is the characteristic map of $2k$-cell in $\Pn$.
Let $[X_{2k},\iota_{2k-1}]_r$ be the relative Whitehead product (defined in \cite{Blakers prodcuts}) of the generator $X_{2k}$ of $\pi_{2k}(\Pn,S^{2k-1}) $ and the generator $\iota_{2k-1}$ of $\pi_{2k-1}( S^{2k-1})$.

There is the following exact sequence of the homotopy groups for the pair $(\Pn, S^{2k-1})$
\begin{align}
	\xymatrix{
		\pi_{4k-2}(\Pn)	 \ar[r]^-{j_\ast}  & \pi_{4k-2}(\Pn,S^{2k-1})\ar[r]^-{\partial} & \pi_{4k-3}(S^{2k-1}).  } \label{exact1 pi(P S)}
\end{align}
Let $(a,b)$ be the greatest common divisor of integers $a$ and $b$ and $n_2=\left\{
\begin{array}{ll}
	n, &\! \hbox{$k=2,4$;} \\
	2n, &\! \hbox{$k\neq 2,4$ and $2|n$.}
\end{array}
\right.$
Based on the conditions satisfied by the attaching map $f$, we can provide a criterion for the complex 
$X$ given in (\ref{cofiber X}) to be homotopy equivalent to the total space of a $S^{2k-1}$-fibration over $S^{2k}$.

The key ingredient in establishing the criterion is a homotopy construction of the map $\mu$ introduced at the beginning of Section \ref{sec:Props and Lems} . The existence of the construction itself implies the identity (\ref{mu}).

\begin{theorem}\label{thm: X fibration}
	Let $X$ be a CW complex given by (\ref{cofiber X}). Then $X$ is homotopy equivalent to the total space of an $S^{2k-1}$-fibration over $S^{2k}$   if and only if  
	$\exists m,\tau\in \mathbb{Z}$ such that the following conditions are true:
	\begin{itemize}
		\item $2\mid n$ if $k\ne 2,4$
		\item $(\tau,n)=1$ 
		\item $j_\ast(f)=m[X_{2k},\iota_{2k-1}]_r$
		\item  $m \equiv \pm \tau^2$ (\text{mod}~ $n_2$)
	\end{itemize}
\end{theorem}

As a corollary, we establish identity (\ref{relative}) for sphere fibrations.

\begin{corollary}\label{corollary: Gener Sasao}
	If $X=\Pn\cup_{f}e^{4k-1}$
	is homotopy equivalent to a total space of  $S^{2k-1}$-fibration over $S^{2k}$, then $X\simeq \Pn\cup_{\tilde{f}} e^{4k-1}$, such that $j_{\ast}(\tilde{f})=[X_{2k},\iota_{2k-1}]_r$.
\end{corollary}

It is known from the main theorem of  \cite{Eulerclass} that  odd integer $n$ can not be realized as the Euler number of a $2k$-dimensional real vector bundle over $S^{2k}$ for $k\neq 2,4$. That is for any  CW complex  $X$ given by (\ref{cofiber X}) for $k\neq 2,4$ and  $n$ is odd, then it is not homotopy equivalent to a $S^{2k-1}$-bundle over $S^{2k}$.
By Theorem \ref{thm: X fibration}, we can extend this conclusion from the case of sphere bundles to the case of sphere fibrations. 
\begin{corollary}
	Let $X$ be a  CW complex given by (\ref{cofiber X}) with  $k\neq 2,4$ and $n$ is odd, then  $X$ is not homotopy equivalent to an $S^{2k-1}$-fibration over $S^{2k}$.
\end{corollary}

Let $K^n_k=Ker(p_{2k\ast}:\pi_{4k-2}(\Pn)\to \pi_{4k-2}(S^{2k}))$.

By \cite[Theorem 4.2]{James On sph bundle}, $\partial \pi_{4k-1}(P^{2k}(n),S^{2k-1})=n \pi_{4k-2}( S^{2k-1})$, it follows from  (\ref{exact1 pi(P S)}) that there is a short exact sequence
\begin{align}
	0 \to  \pi_{4k-2}( S^{2k-1})/n \pi_{4k-2}( S^{2k-1}) \overset{i_{2k-1\ast}}{\to}  K^n_k \overset{j_\ast}{\to} j_\ast(K^n_k) \to 0 \label{exact: K}
\end{align}
By Lemma \ref{lem: j(K)},  $ j_{\ast}(K^n_k)$ is a cyclic group generated by  $[X_{2k},\iota_{2k-1}]_r$  when $k=2,4$ or $k\neq 2,4$ and $2|n$.  Let  $\theta_{k}^n\in	\pi_{4k-2}(\Pn)$ be any fixed lift of $[X_{2k},\iota_{2k-1}]_r$ by the map  $j_{\ast}$. 
Thus 
\begin{align}
	K^n_k=\lr{\theta_{k}^n}+i_{2k-1\ast}\pi_{4k-2}( S^{2k-1})/n i_{2k-1\ast}\pi_{4k-2}( S^{2k-1}). \label{equ: K for general n}
\end{align}
where $\lr{\theta_{k}^n}$ denotes the cyclic subgroup of $K^n_{k}$ generated  by $\theta_{k}^n$. 

Now combine (\ref{equ: K for general n}) and Theorem \ref{thm: X fibration}, we get the following theorem
\begin{theorem}\label{example}
	Let $X$ be a CW complex given by (\ref{cofiber X}). Then $X$ is homotopy equivalent to the total space of an $S^{2k-1}$-fibration over $S^{2k}$   if and only if  
	$\exists a,\tau\in\Z,\gamma\in \pi_{4k-2}(S^{2k-1})$ such that the following conditions are true:
	\begin{itemize}
		\item $2\mid n$ if $k\ne 2,4$
		\item $(\tau,n)=1$ 
		\item $f=a\theta_{k}^n + i_{2k-1}\fhe \gamma$
		\item  $a \equiv \pm \tau^2$ (\text{mod}~ $n_2$)
	\end{itemize}
\end{theorem}

\begin{remark}
	Theorem \ref{example} enables us to count the number $G_k^n$ of  homotopy types of the total spaces of $S^{2k-1}$-fibrations over $S^{2k}$ for specific $k$.
	As an application of Theorem \ref{example}, we get the  $G_k^n$ for $2\leq k\leq 6$ in a  subsequent paper of this topic \cite{ZhuPanclassfyBundle}. 
\end{remark}

It is well known that  the isomorphism classes of  $S^{2k-1}$-bundles over $S^{2k}$ are determined by the characteristic maps in  $\pi_{2k-1}(SO(2k))$.  For $k=2$, based on the results of  homeomorphism  and diffeomorphism classification of the total space of theses sphere bundles by finding the complete invariants,  Crowley provided a homotopy classification by the characteristic maps  \cite[Theorem 1.1]{Crowley}. For $k=4$, by the  same methods  as that of  Crowley,  complete diffeomorphism classification and  partial homeomorphism classification were achieved by Grey \cite{S8bundles}, who also gave a conjecture on the homotopy classification at the end of his Master’s thesis.  In our subsequent paper \cite{ZhuPanclassfyBundle}, by revealing the  relationship between the characteristic map  of $S^{2k-1}$-bundle over $S^{2k}$ and the attaching maps of its total space for $k=2,4$,   we  classify   homotopy types of total spaces of such bundles  via homotopy-theoretic methods by using Theorem  \ref{example} and give an  answer to  the Conjecture 5.5.3 in \cite{S8bundles}.

The previous results also allow us to give a homotopy classification of total spaces $X$ by the attaching maps of the top dimensional cell of the total spaces when $k\ne 2,4$ and thus solve the corresponding Problem (ii).

\begin{theorem}\label{thm: x1=X2}
	 Let $k\neq 2,4$ and   $X_i= \Pn\cup_{f_i}e^{4k-1}$ with $j_{\ast}(f_i)=[X_{2k},\iota_{2k-1}]_r$ ($i=1,2$) be homotopy equivalent to the total space of $S^{2k-1}$-fibration over $S^{2k}$. Then
	$X_1 \simeq X_2$ if and only if there is an integer $t$ such that  $t\Sigma^{\infty}f_1=\Sigma^{\infty}f_2$  with $t^2\equiv 1 $ (mod $2n$). 
\end{theorem}

Cohen \cite[Section 21]{Cohen Course} constructed  elements $\alpha\in \pi_{4k-2}(P^{2k}(2^r))$   which satisfies $\bar h_{2^r}(\hat{\alpha})=[u_{2k-1}, v_{2k}]$, where  $\bar h_{2^r}$ is the mod $2^r$ reduction of Hurewicz homomorphism and $\hat{\alpha}\in \pi_{4k-3}(\Omega P^{2k}(2^r))$ denotes  the adjoint map of $\alpha$. 
Here we call  any such $\alpha$ a ``Cohen element".  Cohen calculated the orders of $\alpha$ when $k\neq 2^i(i\geq 1)$ \cite{Cohen Course} or case  $r=1$ \cite{CohenWu95}.Mukai independently computed the orders of $\alpha$ for the Stiefel manifold $ V_{2k+1,2}$. The following theorem  fills a gap in the existing results concerning the orders of $\alpha$ and  reveals a relation between $\alpha$ and the attaching map of $X=P^{2k}(2^r)\cup_f e^{4k-1}$ which is homotopy equivalent to a total space of $S^{2k-1}$-fibration over $S^{2k}$. 

\begin{theorem}\label{thm:Cohen element} For an abelian group $A$, denote $o(A)=min\{ \text{positive integer}~a|~aA=0\}$ as the order of $A$ and $o_{k}^n=o(K_{k}^n)$.
	
	\begin{enumerate}[(i)]
		\item\label{thm:Cohen element:i} 
		\begin{align*}
			\mathbf{order} (\theta_{k}^n)=o_{k}^n=\left\{
			\begin{array}{ll}
				n, & \hbox{$2\nmid n$ or $k=2, 8|n$ or $k=4, 16|n$;} \\
				4n, & \hbox{$2\nmid k$, and $n\equiv 2$ (mod $4$);}	
				\\
				2n, & \hbox{otherwise.}
			\end{array}
			\right.
		\end{align*}
		Moreover if $2|n$, then  $K_{k}^{n}\cong \Z_{o_{k}^n}\oplus A$ with $2 o(A)|o_{k}^n$;
		\item\label{thm:Cohen element:ii} 
		Let $\alpha\in \pi_{4k-2}(P^{2k}(2^r))$ be any Cohen element,  and $f$ is the attaching map of  $X=P^{2k}(2^r)\cup_f e^{4k-1}$ which is homotopy equivalent to a total space of $S^{2k-1}$-fibration over $S^{2k}$. We have 
		\begin{align*}
			\mathbf{order} (f)=\mathbf{order}(\alpha)=o_{k}^n.
		\end{align*}
		\item\label{thm:Cohen element:iii} If  $f$ is given by (\ref{thm:Cohen element:ii}) , then   
		\begin{align*}
			\bar h_{2}(\hat f)=[u_{2k-1},v_{2k}],
		\end{align*}
		where  $\bar h_{2}$ is the mod $2$ reduction of Hurewicz homomorphism.
	\end{enumerate}
\end{theorem}

\begin{corollary}\label{coro:stable trival of theta}
	There is a  lift $\theta_{k}^n\in K_{k}^n$ of $[X_{2k},\iota_{2k-1}]_r$ in (\ref{equ: K for general n} ) such that $\Sigma^{\infty} \theta_k^n=0$. 
\end{corollary}

Theriault observed that stable triviality of the attaching map is almost equivalent to the existence of the 2-local " homotopy tangent bundle". The above two result generalize this observation. This reflects a deep relation between the existence of $S^{2k-1}$-fibration over $S^{2k}$ and elements of highest order in $\pi_{4k-2}P^{2k}(n)$.

This relation will be discussed in more details in the future including the existence of ``homotopy tangent bundle" which is functorial with respect to certain natural maps between Moore spaces.

This paper is organized as follows: In Section \ref{sec:Htpy Moore}, 
 some homotopy properties of Moore spaces are given, along with the proofs of  Theorem \ref{thm:Cohen element} and Corollary \ref{coro:stable trival of theta}.
 Section \ref{sec:Props and Lems} gives some useful propositions and lemmas for some extensions of maps.   In Section \ref{sec: Homomorphism Lambda }, we define a map from $K_{k}^n$ to $\Z_{n_2}$, which is an epimorphism,  to detect the attaching map $f\in K_{k}^n$ of the top-dimensional cell in an \( S^{2k-1} \)-fibration over \( S^{2k} \) and give the proofs of  Theorem \ref{thm: X fibration} and Corollary \ref{corollary: Gener Sasao}. The last section gives the proof of Theorem \ref{thm: x1=X2}.

\section{Homotopy of Moore spaces}
\label{sec:Htpy Moore}
In this section, we give some properties of Moore spaces and give the proof of Theorem \ref{thm:Cohen element} and Corollary \ref{coro:stable trival of theta}.

Let $Aut(X)$ be the group of self-homotopy equivalences of $X$.  

By Corollary 1.4.10 of \cite{Baues} and Lemma (5) of \cite{SASAO2}, 
$$[P^{k+1}(n), P^{k+1}(n)]=\left\{
\begin{array}{ll}
	\Z_n\{\iota_P\}, & \hbox{$n$  odd;} \\
	\Z_{2n}\{\iota_P\},~\text{with}~i_k\fhe\eta_k\fhe p_{k+1}=n\iota_P,& \hbox{$ 2|n, 4\nmid n$;}
	\\
	\Z_n\{\iota_P\}\oplus \Z_2\{i_k\fhe\eta_k \fhe p_{k+1} \}, & \hbox{$4|n$.}	
\end{array}
\right.$$
where $\iota^{k+1}_{P}$ is the identity map of $P^{k+1}(n)$ and it is simplified by $\iota_{P}$  when no confusion arises.
From Lemma (7) of \cite{SASAO2}, 
$$Aut(P^{k+1}(n))=\left\{
\begin{array}{ll}
	\{t\iota_P | t\in \Z^{\ast}_n\}, & \hbox{$n$  odd;} \\
	\{t\iota_P | t\in \Z^{\ast}_{2n}\},~\text{with}~i_k\fhe\eta_k\fhe p_{k+1}=n\iota_P, & \hbox{$ 2|n, 4\nmid n$;}
	\\
	\{ t\iota_P+\epsilon i_k\fhe\eta_k\fhe p_{k+1}~|~t\in \Z_{n}^{\ast},\epsilon\in \{0,1\}\}, & \hbox{$4|n$.}	
\end{array}
\right.$$

The following proposition comes from \cite[Lemma 4.10]{Yamaguchi}
\begin{proposition}\label{Prop: hty equi}
	$	\Pn\cup_{f_1}e^{4k-1}\simeq 	\Pn\cup_{f_2}e^{4k-1}$ if and only if there is a homotopy equivalence $g\in  Aut(P^{2k}(n))$ such that $g\fhe f_1=\pm f_2$.
\end{proposition} 

Note that $K^n_k=Ker(p_{2k\ast}:\pi_{4k-2}(\Pn)\to \pi_{4k-2}(S^{2k}))$, then we have 
\begin{lemma}\label{lem: j(K)}
	\begin{align}
		j_{\ast}(K^n_k)=\left\{
		\begin{array}{ll}
			\Z_n\{[X_{2k}, \iota_{2k-1}]_r\}, & \hbox{$k=2,4$;} \\
			\Z_{2n}\{[X_{2k}, \iota_{2k-1}]_r\}, & \hbox{$ k\neq 2,4$ and $ 2|n$;}\\
			\Z_{n}\{2[X_{2k}, \iota_{2k-1}]_r\}, & \hbox{ $k\neq 2,4$ and $2\nmid n$.}
		\end{array}
		\right.
	\end{align}
\end{lemma}
\begin{proof}
	By the following commutative diagram of exact sequences:
	
	\begin{align}
		\xymatrix{
			K^n_k\ar[r]\ar[d]_{j_\ast }&	\pi_{4k-2}(\Pn) \ar[d]_{j_\ast } \ar[r]  & \pi_{4k-2}(S^{2k})\ar[d]^{\cong} \\
			Ker(p_{2k\ast}) \ar[r]&	\pi_{4k-2}(\Pn,S^{2k-1}) \ar[r]^-{p_{2k\ast}}  & \pi_{4k-2}(S^{2k},\ast)  } \label{diam: K}
	\end{align}
	where  $p_{2k}:(\Pn,S^{2k-1})  \to (S^{2k},\ast) $ is the map induced by the pinch map $p_{2k}$.
	
	From the above commutative diagram and (\ref{exact1 pi(P S)}), it is easy to get 
	\begin{align}
		j_\ast(K^n_k)= Ker(p_{2k\ast})\cap Ker(\partial). \label{equ: i(K)}
	\end{align}
	By the result of Sasao \cite{SASAO1}, $Ker(p_{2k\ast})=\Z_{n}\{[X_{2k},\iota_{2k-1}]_r\}$ or $\Z_{2n}\{[X_{2k},\iota_{2k-1}]_r\}$ according as $k=2,4$ or not. 
	Now $\partial ([X_{2k},\iota_{2k-1}]_r)=-[\partial X_{2k},\iota_{2k-1}]=-n [\iota_{2k-1},\iota_{2k-1}]$ by \cite[(3.5)]{Blakers prodcuts} or \cite[(2.1)]{James}. 
	Therefore, Lemma \ref{lem: j(K)} is obtained from (\ref{equ: i(K)}) since $[\iota_{2k-1}, \iota_{2k-1}]=0$ for $k=2,4$ and the order of $[\iota_{2k-1}, \iota_{2k-1}]$ is 2 for $1 <k\neq 2,4$ by \cite[Theorem 8.8]{GTM61}.
\end{proof}

 We give the proof of Theorem \ref{thm:Cohen element} and Corollary \ref{coro:stable trival of theta}. 

\begin{proof}[Proof of Theorem \ref{thm:Cohen element} (\ref{thm:Cohen element:i})]
	$\mathbf{order}(\theta_{k}^n)=o_{k}^n$ is from (\ref{equ: K for general n}) and Lemma \ref{lem: j(K)}. 
	
	The results  for $k=2,4$ are obtained by Theorem (iii) of \cite{SASAO2}. 
	
	Next we  prove it for $k\neq 2,4$.	By \cite[Proposition]{Barratt}, 
	\begin{align*}
		&	p^r\pi_{4k-2}(P^{2k}(p^r))=0, p ~\text{is odd prime, ~~and~~~ } 2^{r+1}\pi_{4k-2}(P^{2k}(2^r))=0, r\geq 2.
	\end{align*}
	This implies $o_{k}^n=n$ for odd $n$ and   $o_{k}^n=2n$ for $k\neq 2,4$ with $4|n$ by  exact sequence (\ref{exact: K}) and Lemma \ref{lem: j(K)}.   Since  $4\pi_{4k-2}(P^{2k}(2))=0$ for $2|k$ by \cite[Corollary 1.2]{Mikhailov JieWu},  $o_{k}^n=2n$ for  $2|k$ and  $n\equiv 2$ (mod $4$). 
	From  \cite[Theorem 2.2]{CohenWu95}, there is an order $8$ element $\lambda_{2k}\in K_{k}^2$ for $2\nmid k$. Hence  $o_{k}^n=4n$ for $2\nmid k$ and  $n\equiv 2$ (mod $4$).

	Now, use the  exact sequence (\ref{exact: K}) and Lemma \ref{lem: j(K)} again, we get the ``Moreover" part of Theorem \ref{thm:Cohen element} (\ref{thm:Cohen element:i}) for $k\neq 2,4$.
\end{proof}

\begin{proof}[Proof of Theorem \ref{thm:Cohen element} (\ref{thm:Cohen element:ii}) and (\ref{thm:Cohen element:iii})]
	We first prove that any Cohen element $\alpha\in K_{k}^{2^r}$. 
	
	Since  $E^{\infty}\fhe f_{r}= 0$ where $f_r$ comes from fibration sequence (\ref{fib:X_r}) and $E^{\infty}: P^{2k}(2^r)\rightarrow QP^{2k}(2^r)$ is the stabilization map, there is a lift  $g:P^{2k}(2^r)\rightarrow X_r$ such that $\mathbf{i}_r\fhe g= E^{\infty}$ and it induces isomorphism on $2n-1$-dimensional homology groups. Let $F\xrightarrow{j'} P^{2k}(2^r)\xrightarrow{g} X_r$ be the fibrations sequence where  $F$ is the homotopy fibre of $g$.  As mentioned in the proof of \cite[Lemma 21.1]{Cohen Course},  $[u_{2k-1},v_{2k}]$ is in the image of $(\Omega j')_{\ast}: H_{4k-3}(\Omega F;\Z_{2^r})\rightarrow  H_{4k-3}( \Omega P^{2k}(2^r);\Z_{2^r})$ and $\Omega F$ is $(4k-4)$-connected. Thus there is a map $\rho: S^{4k-2}\rightarrow  F$ such that $\bar h_{2^r}((\Omega j')\fhe \hat\rho)=[u_{2k-1},v_{2k}]$. 
	Then $\alpha$ is defined by $j'\fhe \rho$.  Consider the following homotopy commutative diagram 
	\begin{align*}
		\xymatrix{
			S^{4k-2}\ar@/_0.5pc/[rr]_{\alpha}\ar[r]^{\rho}&	F\ar[r]^-{j'} & P^{2k}(2^r)\ar[d]^{p_{2k}}\ar[rr]^-{E^{\infty}=\mathbf{i}_r\fhe g} && Q P^{2k}(2^r) \ar[d]^{^{Qp_{2k}}}\ar[r]^{f_{r}} &K(\Z_2, 4k-2) \\
			&&S^{2k}\ar[rr]^-{E^{\infty}}& &QS^{2k}&
		} 
	\end{align*}
	we get $E^{\infty}\fhe p_{2k}\fhe \alpha= (Qp_{2k})\fhe \mathbf{i}_r\fhe g\fhe j'\fhe \rho=0$. Since $E^{\infty}:S^{2k}\rightarrow QS^{2k}$ is $(4k-1)$-connected, $p_{2k}\fhe \alpha\simeq \ast$, i.e., $\alpha\in K_{k}^{2^r}$.

	Secondly, the order of $\alpha$ must be $o_{k}^{2^r}$. Otherwise,  $\alpha=2t\theta_{k}^{2^r}+i_{2k-1}\fhe \gamma$ for some integer $t$ and $\gamma\in\pi_{4k-2}(S^{2k-1})$ by (\ref{equ: K for general n}) . Then 
	\begin{align*}
		\bar h_{2^r}(\hat \alpha)= 2t\bar h_{2^r}(\hat \theta_{k}^{2^r})+\bar h_{2^r}((\Omega i_{2k-1})\fhe \hat \gamma)= 2t\bar h_{2^r}(\hat \alpha)
	\end{align*}
	which  contradicts to $\bar h_{2^r}(\hat \alpha)=[u_{2k-1},v_{2k}]$ since the order of  $[u_{2k-1},v_{2k}]$ is $2^r$.

	If $f$ is the attaching map of $X$ given in Theorem \ref{thm:Cohen element}, then $\Lambda(f)=\pm \tau^2$ with $\tau\in \Z^{\ast}_{n_2}$ by  Theorem  \ref{thm: condition on Hf}. From  Lemma \ref{lem:homo bar H}, the order of $f$ is also $o_{k}^{2^r}$. The  proof of  Theorem \ref{thm:Cohen element} (\ref{thm:Cohen element:ii})  is finished. 
	
	At last, from equation (\ref{equ: K for general n})  and Theorem \ref{thm:Cohen element} (\ref{thm:Cohen element:i}) (\ref{thm:Cohen element:ii}), we get
	\begin{align*}
		f=t\alpha+i_{2k-1}\xi, ~~\text{for some odd integer}~t~\text{and}~\xi\in \pi_{4k-2}(S^{2k-1}).
	\end{align*}
	Then Theorem \ref{thm:Cohen element} (\ref{thm:Cohen element:iii})  is easily obtained.
\end{proof}

\begin{proof}[Proof of Corollary \ref{coro:stable trival of theta}]
	Since $K_k^n$ is a finite group and its $p$-primary component   is  $K_k^{p^r}$ for any prime $p$,	we only need to show that there is a lift $\theta_{k}^{p^r}\in  K_k^{p^r}$  of $[X_{2k},\iota_{2k-1}]_r$  for any prime $p$ such that  $\Sigma^{\infty} \theta_k^{p^r}=0$. 
	
	Firstly,  for any prime $p$, fix a lift $\theta_{k}^{p^r}\in  K_k^{p^r}$  of $[X_{2k},\iota_{2k-1}]_r$.

	We have the following commutative diagram of exact sequences
	\begin{align}
		\xymatrix{
			\pi_{4k-2}^s(S^{2k-1};p)\ar[r]^-{i_{2k-1\ast}^s}& Ker(p^s_{2k\ast})\ar@{>->}[r]&	\pi^s_{4k-2}(P^{2k}(p^r))  \ar[r]^{p_{2k\ast}^s}  & \pi_{4k-2}^s(S^{2k};p) \\
			\pi_{4k-2}(S^{2k-1};p)\ar[u]^{\Sigma^{\infty}}\ar[r]^-{i_{2k-1\ast}}&K_{k}^{p^r} \ar[u]^{\Sigma^{\infty}} \ar@{>->}[r]&	\pi_{4k-2}(P^{2k}(p^r)) \ar[r]^-{p_{2k\ast}} \ar[u]^{\Sigma^{\infty}} & \pi_{4k-2}(S^{2k};p)\ar[u]^{\Sigma^{\infty}\cong}  } 
	\end{align}
	By $p_{2k}\fhe \theta_{k}^{p^r}=0$, we get $\Sigma^{\infty}\theta_{k}^{p^r}\in Ker(p_{2k\ast}^s)=Im(i_{2k-1\ast}^s)$

	If $p$ is odd, then $	\pi_{4k-2}(S^{2k-1};p)\xrightarrow{\Sigma^{\infty}}	\pi_{4k-2}^s(S^{2k-1};p)$ is surjective by \cite[(13.5)]{Toda}. So is $K_{k}^{p^r}\xrightarrow{\Sigma^{\infty}} Ker(p_{2k\ast}^s)$.  Hence 
	\begin{align*}
		\Sigma^{\infty}\theta_{k}^{p^r}=i_{2k-1\ast}^s\Sigma^{\infty}\gamma_{p},~~\text{for some }~\gamma_{p}\in \pi_{4k-2}(S^{2k-1};p). 
	\end{align*}
	Thus $\Sigma^{\infty}(\theta_{k}^{p^r}-i_{2k-1}\fhe\gamma_{p})=0$. Replacing $\theta_{k}^{p^r}$ by $\theta_{k}^{p^r}-i_{2k-1}\fhe \gamma_{p}$, we get the $\theta_{k}^{p^r}$ satisfied the condition.
	
	If $p=2$,  consider the $(4k-1)$-skeleton of $X_r$ given in the homotopy fibration (\ref{fib:X_r}). It is a total space of  $S^{2k-1}$-fibration over $S^{2k}$ under $2$-localization, which has homotopy type $P^{2k}(2^r)\cup_{f_{\tau}} e^{4k-1}$.
	From \cite[page 375]{Theriault}, we have  the  following cofibration sequence
	\begin{align*}
		S^{4k-2}\xrightarrow{f_{\tau}} P^{2k}(2^r) \xrightarrow{i_\tau} 	\tau_r(S^{2k})
	\end{align*}
	and  the  homotopy commutative diagram  with $f_r\fhe E^{\infty}=0$.
	\begin{align}
		\xymatrix{	\tau_r(S^{2k})\ar@{^{(}->}[r]^-{j_{\tau}} &  X_r \ar[r]^-{\mathbf{i}_r} &Q(P^{2k}(2^r))\ar[r]^{f_r}&K(\Z_2, 4k-2) \\
			&&P^{2k}(2^r)\ar[u]_-{E^{\infty}}\ar[ur]_-{0} \ar[ull]^{i_{\tau}}&  } \label{diam:tau_r}
	\end{align}
	So $E^{\infty}\fhe f_{\tau}=\mathbf{i}_r\fhe j_{\tau}\fhe i_{\tau}\fhe f_{\tau}=0$. By Theorem \ref{thm:Cohen element}, $\theta_{k}^{2^r}=tf_{\tau}+i_{2k-1}\fhe \gamma_2$ for some  $\gamma_2\in \pi_{4k-2}(S^{2k-1};2)$. Replacing $\theta_{k}^{2^r}$ by $\theta_{k}^{2^r}-i_{2k-1}\fhe \gamma_{2}$, we get the $\theta_{k}^{2^r}$ satisfied the condition.
	
\end{proof}

By the application of Theorem \ref{thm:Cohen element}, we get the following lemmas about Moore spaces, which will be useful for proving Theorem \ref{thm: x1=X2}.

\begin{lemma}\label{Lem: pH2(theta), k=3,5}
	For $k\geq 2$ and $2|n$,  	$(p\wedge \iota^{2k-1}_P)\fhe H_2(\theta_k^{n})=li_{4k-2}$ for some odd integer $l$.
\end{lemma}
\begin{proof}
	We only prove this lemma for the case $k\neq 2,4$, since the cases 
	$k=2,4$ are similar and easier. To prove $l$ is odd, it suffices to consider the $2$-localization 
	$(P^{2k}(n))_{(2)}\simeq P^{2k}(2^r)$, i.e, we prove this lemma for the case $n=2^r$, $r\geq 1$.
	
	We generalize the map $\bar H_2: \Omega P^{n+1}(2)\rightarrow \Omega P^{2n+1}(2)$ in  \cite[Section 4.4]{WJProjplane} to the mod $2^r$ Moore space as follows
	
	$\bar H_2: \Omega P^{n+1}(2^r)\simeq J(P^{n}(2^r))\xrightarrow{H'_2}  J(P^{n}(2^r)\wedge P^{n}(2^r))\xrightarrow{J(q)}  J(P^{2n}(2^r)) \simeq \Omega  P^{2n+1}(2^r)$.
	\newline
	where $J(X)$ is the James construction for some based spase $X$, which is homotopy equivalent to the $\Omega\Sigma X$ when $X$ is $1$-connected; $q=p_{n}\wedge \iota_P$ is the pinch map $P^{n}(2^r)\wedge P^{n}(2^r)\rightarrow P^{n}(2^r)\wedge P^{n}(2^r)/(S^{n-1}\wedge P^{n}(2^r))=P^{2n}(2^r)$;  $ \Omega P^{n+1}(2^r)\simeq J(P^{n}(2^r))\xrightarrow{H'_2}  J(P^{n}(2^r)\wedge P^{n}(2^r))\simeq \Omega\Sigma(P^{n}(2^r)\wedge P^{n}(2^r))$ is  the second James-Hopf invariant. 
	
	By Lemma 3.12 of \cite{Steer1963}, the second James-Hopf invariant $H_2'$ given here is   homotopic to the Hopf-Hilton invariant $H_2$ defined in  \cite[(4.1)]{Hopfinvar}. Consider the composition
	\begin{align*}
		\bar{H}_{2\ast}: H_{\ast}(\Omega P^{2k}(2^r),\Z_2)\!\xrightarrow{H'_{2\ast}}\!H_{\ast}(\Omega \Sigma P^{2k-1}(2^r)\!\wedge\! P^{2k-1}(2^r),\Z_2)\!\xrightarrow{(\Omega\Sigma q)_{\ast}}\!H_{\ast}(\Omega P^{4k-1}(2^r),\Z_2)	
	\end{align*}
	where  $H_{\ast}(\Omega P^{4k-1}(2^r);\Z_2)=T(u',v')$ and the degrees of $u'$ and $v'$ are $4k-3$ and $4k-2$ respectively.
	By \cite[Proposition 5.3]{CohenTalor} we get $\bar{H}_{2\ast}([u,v])=u'$ and $\bar{H}_{2\ast}(v^2)=v'$.  
	Hence  $\bar H_{2\ast}\bar h_{2}(\hat\theta_k^{2^r})=u'$ by Theorem \ref{thm:Cohen element} (\ref{thm:Cohen element:iii}). This implies the following composition is homotopy equivalent to $l\hat i_{4k-2}:S^{4k-3}\rightarrow \Omega P^{4k-2}(2^r)$ with odd $l$
	\begin{align}
		&	\xymatrix{
			S^{4k-3}\ar@/_0.5pc/[rrr]_{\hat \theta_k^{2^r}}	\ar@{^{(}->}[rr]^-{\Omega\Sigma=E_{S^{4k-3}}}&& \Omega S^{4k-2}\ar[r]^-{\Omega\theta_k^{2^r}}
			&	\Omega P^{2k}(2^r)\ar[r]^-{\bar H_2}	&\Omega P^{4k-1}(2^r)}, \nonumber \\
		i.e., &~ ~~~~~~~~~~~~~\bar H_2\fhe \hat\theta_k^{2^r} =l \hat i_{4k-2}, ~l~\text{is odd},	\label{Equ: H2Omega(theta)}
	\end{align}
	where 	$\kappa: [\Sigma X,Y]\xrightarrow{\cong} [X,\Omega Y]$, $f\mapsto \hat f=(\Omega f)\fhe E_X$ and  $E_X$ denotes the canonical inclusion $X\rightarrow \Omega\Sigma X$.
	
	Recall the  Hopf invariant $H_2:\pi_{4k-2}(P^{2k}(2^r))\rightarrow \pi_{4k-2}(P^{2k}(2^r)\wedge P^{2k-1}(2^r))$ is defined by the following composition
	\newline
	$\pi_{4k-2}(P^{2k}(2^r))\xrightarrow{\kappa \cong }	\pi_{4k-3}(\Omega P^{2k}(2^r))\xrightarrow{H'_2}\pi_{4k-3}(\Omega P^{2k}(2^r)\wedge P^{2k-1}(2^r))\xrightarrow{\kappa^{-1}}\pi_{4k-2}( P^{2k}(2^r)\wedge P^{2k-1}(2^r))$. 
	\begin{align}
		\text{That is},~	&	H_2(\theta_k^{2^r})=\kappa^{-1}( H'_2\fhe\hat\theta_k^{2^r})=\kappa^{-1}( H'_2\fhe\Omega(\theta_k^{2^r})\fhe E_{S^{4k-3}}). \label{Equ 1:H2(theta)}\\
		li_{4k-2} &=\kappa^{-1}(l\hat i_{4k-2})=\kappa^{-1}(\bar  H_2\fhe\hat\theta_k^{2^r}) ), ~~\text{by   (\ref{Equ: H2Omega(theta)})}\nonumber  \\
		&=\kappa^{-1}((\Omega\Sigma q)\fhe   H'_2\fhe\Omega(\theta_k^{2^r})\fhe E_{S^{4k-3}} ),~~\text{by ~definition of $\bar H_2$}\nonumber\\
		&=\kappa^{-1}(\Omega  (p_{2k}\wedge \iota_P)\fhe  \kappa(H_2(\theta_k^{2^r})) ),~~\text{by definition of  $q$  and (\ref{Equ 1:H2(theta)})}\nonumber\\
		&=(p\wedge \iota^{2k-1}_P)\fhe H_2(\theta_k^{2^r}).\nonumber
	\end{align}
	We complete the proof of this lemma.
\end{proof}

\begin{corollary}\label{cor: pH2(theta), k=3,5}
	Let $k\geq 2$ and $2|n$. for any $f=a\theta_k^{n}+i_{2k-1}\fhe \gamma\in K_{k}^n$ for some integer $a$ and $\gamma\in \pi_{4k-2}(S^{2k-1})$, 	$( \iota^{2k}_P\wedge p_{2k-1})\fhe H_2(f)=ali_{4k-2}$ for some odd integer $l$,	where $H_2$ is the Hilton-Hopf invariant given by (4.1) of \cite{Hopfinvar}.
\end{corollary}
\begin{proof}
	If $k\geq 2$ and $2|n$, then  we get  $( \iota^{2k}_P\wedge p_{2k-1})\fhe H_2(\theta_k^{n})=li_{4k-2}$ for some odd integer $l$ by Lemma \ref{Lem: pH2(theta), k=3,5}.
	\begin{align*}
		( \iota^{2k}_P\wedge p_{2k-1})\fhe H_2(f)=& a(\iota^{2k}_P\wedge p_{2k-1})\fhe H_2(\theta_k^{n})+(\iota^{2k}_P\wedge p_{2k-1})\fhe H_2(i_{2k-1}\fhe \gamma)\\
		=&ali_{4k-2}+(\iota^{2k}_P\wedge p_{2k-1})\fhe(\Sigma i_{2k-2}\wedge i_{2k-1})\fhe H_2(\gamma)\\
		=&ali_{4k-2}.
	\end{align*}
\end{proof}

\begin{lemma}\label{lem:exist g'}
	Let $k\geq 2$ and $2|n$. 	If  $g\in Aut(\Pn)$ and $f=\theta_k^{n}+i_{2k-1}\fhe \gamma\in K_{k}^n$ for some $\gamma\in \pi_{4k-2}(S^{2k-1})$, then 
	there is a $g'\in Aut(\Pn)$, such that $g'\fhe f=g\fhe f+i_{2k-1}\fhe [\eta_{2k-1},\iota_{2k-1}]$.	
\end{lemma}
\begin{proof}
	Note that $g=t\iota_P+\epsilon_1 i_{2k-1}\fhe\eta_{2k-1}\fhe p_{2k}$ for some odd integer $t$ and $\epsilon_1\in \{0,1\}$. 
	Let $g'=g+i_{2k-1}\fhe\eta_{2k-1}\fhe p_{2k}$. We have 
	\begin{align*}
		g'\fhe f=&g\fhe f+[g,  i_{2k-1}\fhe\eta_{2k-1}\fhe p_{2k}]\fhe H_{2}(f)~(\text{by}~p_{2k}\fhe f=0)\\
		=& g\fhe f+[g, i_{2k-1} \fhe\eta_{2k-1}]\fhe(\iota^{2k}_P\wedge p_{2k-1})\fhe  H_{2}(f )\\
		=&g\fhe f+[g,  i_{2k-1}\fhe\eta_{2k-1}]\fhe i_{4k-2}~\text{(by Corollary \ref{cor: pH2(theta), k=3,5})}\\
		=&g\fhe f+ [t\iota_P+\epsilon_1 i_{2k-1}\fhe\eta_{2k-1} \fhe p_{2k},  i_{2k-1}\fhe\eta_{2k-1}]\fhe (\Sigma i_{2k-2}\wedge \iota_{2k-1}) \\
		=&g\fhe f+ [\iota_P,  i_{2k-1} \fhe\eta_{2k-1}]\fhe(\Sigma i_{2k-2}\wedge \iota_{2k-1}) +\epsilon_1 i_{2k-1}\eta_{2k-1}\fhe[ p_{2k},  \iota_{2k-1}]\fhe (\Sigma i_{2k-2}\wedge \iota_{2k-1}) \\
		=&g\fhe f+ [\iota_{P}\fhe i_{2k-1},  i_{2k-1}\fhe \eta_{2k-1}\fhe\iota_{2k}]+\epsilon_1 i_{2k-1}\fhe\eta_{2k-1}\fhe[ p_{2k}\fhe i_{2k-1},  \iota_{2k}]\\
		=&g\fhe f+[ i_{2k-1},  i_{2k-1} \fhe\eta_{2k-1}]\\
		=&g\fhe f+i_{2k-1}\fhe [\eta_{2k-1},\iota_{2k-1}].
	\end{align*}
\end{proof}

\section{Some propositions for  extensions}
\label{sec:Props and Lems}

 For $ X= P^{2k}(n) \cup_{f}e^{4k-1}$ with $p_{2k}\fhe f=0$, If $p : \Pn\rightarrow S^{2k}$ is any map inducing an epimorphism in cohomology, then $p=ap_{2k}$ with $(a,n)=1$, then $p\circ f=[a]\fhe p_{2k}\fhe  f= [a]\fhe( p_{2k}\fhe  f)=0$.  
So there is an extension $\pi:X\rightarrow S^{2k}$ of $p$ via $i_X$, i.e. $i_X\fhe\pi=p$. In other words,  there is the extension $f_a(\pi)$ satisfies  the following homotopy commutative diagram
\begin{align}
	\small{\xymatrix{
			P^{2k}(n)\ar[d]^-{a\iota_P} \ar[r]^-{i_X}&X\ar[r]^-{p_X}\ar[d]^{\pi}&S^{4k-1}\ar[d]^{f_a(\pi)}\\
			P^{2k}(n) \ar[r]^-{p_{2k}}& S^{2k}\ar[r]^{[n]} & S^{2k}
	} }. \label{digm:fa[pi]}
\end{align}
where $\iota_P$ denotes the homotopy class of the identity map of $\Pn$. 

If $a=1$, then we simplify the notation $f_{a}(\pi)$ by $f(\pi)$.

\begin{lemma} \label{Lem:Hopf inv relation}
For $ X= P^{2k}(n) \cup_{f}e^{4k-1}$ with $p_{2k}\fhe f=0$,	if  $\tilde{\pi}:X\rightarrow S^{2k}$ is an  extension of $ap_{2k}$ with $(a,n)=1$,  then any extension of $ap_{2k}$ is obtained by $\alpha\tilde{\pi}~(\alpha \in\pi_{4k-1}(S^{2k}))$ which is defined by
\begin{align}
	\small{\xymatrix{
			X  \ar[d]_{\alpha \tilde{\pi}} \ar[r]  & S^{4k-1}\vee X  \ar[dl]^{(\alpha ,\tilde{\pi})}  \\
			S^{2k} & 
	} }.  \nonumber
\end{align}
where the horizontal map is the coaction induced by defining cofibration of $X$.

Moreover,  for the  Hopf invariant $H:\pi_{4k-1}(S^{2k})\rightarrow \Z$,
\begin{align}
	H(f_a(\alpha \tilde{\pi}))=n^2H(\alpha)+ H(f_a(\tilde{\pi})). \label{Equ Hopf relation}
\end{align}
\end{lemma}
\begin{proof}
The first conclusion is easily obtained by the following exact sequence
\begin{align}
	[P^{2k}(n), S^{2k}]\xleftarrow{i_X^{\ast}} [X,S^{2k}]\xleftarrow{p_X^{\ast}} \pi_{4k-1}(S^{2k})\xleftarrow{(\Sigma f)^{\ast}} [P^{2k+1}(n), S^{2k}]. \label{exact: i_X}
\end{align}
For the ``Moreover" part,  note that  $\alpha \tilde{\pi}$  fits into the following commutative diagrams:
\begin{align}
	\small{\xymatrix{
			X  \ar[d]_{\alpha \tilde{\pi}} \ar[r]  & S^{4k-1}\vee X \ar[r]^-{id \vee p_X} \ar[d]^-{(\alpha , \tilde{\pi})} \ar[r] & S^{4k-1}\vee S^{4k-1} \ar[d]^{([n]\fhe \alpha , f_a(\tilde{\pi}))}  \\
			S^{2k} \ar[r]_{=}& S^{2k} \ar[r]^{[n]}&  S^{2k} 
	} }.  \label{diam: alpha tidlepi}
\end{align}
Since the composition of the horizontal maps in the top of the  diagram (\ref{diam: alpha tidlepi}) is $\nu\fhe p_X$ \cite[Theorem 4.3.7]{Arkowitz},  where $\nu$ is the co-H-space structure map on $S^{4k-1}$.  Thus we have the following homotopy commutative diagram 
\begin{align*}
	\small{\xymatrix{
			X   \ar[r]^{p_X}\ar[d]_{\alpha \tilde{\pi}} \ar[r]  & S^{4k-1}\ar[d]^{[n]\fhe \alpha + f_a(\tilde{\pi})}  \\
			S^{2k} \ar[r]^{[n]}&  S^{2k} 
	} }. 
\end{align*}
This  implies that $f_a(\alpha \tilde{\pi})$ can be choosen to be $[n]\fhe \alpha + f_a(\tilde{\pi})$.

It follows that  $H(f_a(\alpha \tilde{\pi}))=H([n]\fhe \alpha) + H(f_a(\tilde{\pi}))=n^2H(\alpha)+ H(f_a(\tilde{\pi}))$.

\end{proof}
\begin{remark}\label{rem: diff f(pi)}
The different choices of the extensions $f_a(\tilde{\pi})$ of the map $\tilde{\pi}$ have the same Hopf invariant since they are related by an addition of an element $S^{4k-1} \overset{\Sigma f}{\to} P^{2k+1}(n)\to S^{2k}$ which is a suspension.
\end{remark}

Recall that $H^\ast (\Omega S^{2k}) $ is the free abelian group generated by $\{z_{(2k-1)j} ~|~ j = 0,1,2,\cdots \}$.

\begin{proposition}\label{differential}
	Let $f : S^{4k-1}\to S^{2k} , k > 1$ be a map and G be the homotopy fiber of $f$. Then in the induced fibration sequence $\Omega S^{2k}\to G \to S^{4k-1}$,  
	we have the
	identity $d_{4k-1}(z_{4k-2}) = H(f) ·\bar y_{4k-1}$,	where $\bar y_{4k-1} \in  H^{4k-1}(S^{4k-1})$ is a generator.  
\end{proposition}
\begin{proof}
	Consider the commutative diagrams:
	\begin{align*}
		\footnotesize{\xymatrix{
				\pi_{4k-1}(S^{4k-1}) \ar[d]^-{h} & \pi_{4k-1}(G,\Omega S^{2k})  \ar[d]^{h}\ar[l] \ar[r]^-{\partial} & \pi_{4k-2}(\Omega S^{2k})  \ar[d]^-{h} \\
				H_{4k-1}(S^{4k-1}) & H_{4k-1}(G,\Omega S^{2k})   \ar[l]  \ar[r]^{\partial} & H_{4k-2}(\Omega S^{2k})  }}
	\end{align*}

	where the bottom row gives the transgression of the homology Serre spectral sequence associated to $\Omega S^{2k}\to G \to S^{4k-1}$.
	
	A simple connectivity computation shows that the lower left horizontal map is an isomorphism, after identification, the diagram can be reduced to
	\begin{align*}
		\footnotesize{\xymatrix{
				\pi_{4k-1}(S^{4k-1}) \ar[d]^-{h}  \ar[rr]^{\partial} && \pi_{4k-2}(\Omega  S^{2k})   \ar[d]^-{h}  \\
				H_{4k-1}(S^{4k-1}) \ar[rr]^{d_{4k-1}} && H_{4k-2}(\Omega  S^{2k})   }}
	\end{align*}
	where the bottom map is just the differential $d_{4k-1}$ by Proposition 1.13 of \cite{HatcherSpectr}.
	
	On the other hand, it is well known that $\partial:\pi_{4k-1}(S^{4k-1})\to \pi_{4k-2}(\Omega  S^{2k}) $ is given by $(\Omega f )_\ast: \pi_{4k-2}(\Omega S^{4k-1})\to \pi_{4k-2}(\Omega  S^{2k}) $. 
	Proposition 1.30 of \cite{HatcherSpectr}  and the fact that the Hurewicz homomorphisms in the above diagram are all isomorphic imply that  $d_{4k-1}$ sends a generator to $H(f)$ times a generator.
	
	Dually, this gives the desired equality for differential in the cohomology Serre spectral sequence associated to $\Omega S^{2k}\to G \to S^{4k-1}$.
\end{proof}

If  $\tilde{\pi}:X\rightarrow S^{2k}$  is an  extension  of $ap_{2k}$ with $(a,n)=1$, let $F_{\tilde{\pi}}$ be the homotopy fiber of $\tilde{\pi}$. We  calculate the cohomology of
$F_{\tilde{\pi}}$ using the Serre spectral sequence for the fibration: $\Omega S^{2k} \to F_{\tilde{\pi}} \to X$.  So we have 
\begin{align*}
	E_2^{p,q} = H^p(X) \otimes  H^q(\Omega S^{2k})=\left\{
	\begin{array}{ll}
		\Z\{z_{(2k-1)j}\}, & \hbox{$p=0, q=(2k-1)j$;} \\
		\Z_n\{y_{2k}\otimes z_{(2k-1)j}\}, & \hbox{$p=2k, q=(2k-1)j$;} \\
		\Z\{y_{4k-1}\otimes z_{(2k-1)j}\}, & \hbox{$p=4k-1,q=(2k-1)j$;}\\
		0, & \hbox{otherwise.}
	\end{array}
	\right.
\end{align*}
where  $j=0,1,2,\cdots, $  and  for $j=0$, $z_{0}=z_{(2k-1)j}$ denotes the generator of $H^{0}(\Omega S^{2k})$;
$y_{2k}=i^{\ast -1}_{X}p_{2k}^{\ast}(\bar \iota_{2k})$   and $y_{4k-1}=p^{\ast}_{X}(\bar\iota_{4k-1})$ are generators of $ H^{2k}(X)$ and $ H^{4k-1}(X)$ respectively, where $\bar\iota_{m}$ is the cohomology class of $S^{m}$ by the identity map.  We  simplify the generators  $y_{2k}\otimes z_0$  and $y_{4k-1}\otimes z_0$ by   $y_{2k}$  and $y_{4k-1}$ respectively in the following text.

\begin{proposition}\label{prop1}
	In the Serre spectral sequence for the fibration $\Omega S^{2k} \to F_{\tilde{\pi}} \to X$, we have
	\begin{align*}
		d_{2k}(z_{(2k-1)j}) =\pm a (y_{2k} \otimes  z_{(2k-1)(j-1)}).
	\end{align*}
\end{proposition}
\begin{proof}
	Consider the diagram of fibrations:
	\begin{align}
		\small{\xymatrix{
				\Omega S^{2k} \ar[d] \ar[r]^{=}  & \Omega S^{2k} \ar[d]  \\
				F_{\tilde{\pi}} \ar[r] \ar[d]  & \ast\ar[d] \\
				X\ar[r]^{\tilde{\pi}} &S^{2k}
		} }.  \nonumber
	\end{align}
	Then the  proposition is an easy consequence of the naturality of the Serre
	spectral sequence with respect to maps of fibrations with the help from the knowledge of ring structure of $H^\ast(\Omega S^{2k})$.
\end{proof}

\begin{proposition}\label{prop2}
	The Hopf invariant, $ H(f_a(\tilde{\pi}))$, is a multiple of $n$, i.e.
	$ H(f_a(\tilde{\pi})) = \lambda n$, and in  the Serre spectral sequence for the fibration $\Omega S^{2k} \to F_{\tilde{\pi}} \to X$, we have
	\begin{align}
		d_{4k-1}(nz_{(2k-1)j}) =\lambda  y_{4k-1} \otimes  z_{(2k-1)(j-2)},~~ j=2,3,\cdots.  \nonumber
	\end{align}
\end{proposition}
\begin{proof}
	It follows from Proposition \ref{prop1} that the classes, $nz_{(2k-1)j} \in E_{2k}^{0,(2k-1)j}$
	, survive to the next stage since $ny_{2k}=0$. It
	remains to calculate $d_{4k-1}(nz_{(2k-1)j})$. Let $G$ be the homotopy fiber of $f(\tilde{\pi})$. Using
	the diagram after defining the extension $\tilde{\pi}$, we get a diagram of fibrations:
	\begin{align}
		\small{\xymatrix{
				\Omega S^{2k} \ar[d]\ar[r]^{\Omega [n]}  & \Omega S^{2k} \ar[d]  \\
				F_{\tilde{\pi}} \ar[r]\ar[d] & G \ar[d] \\
				X\ar[r]^{p_X} &S^{4k-1}
		} }.  \label{diam: fibr F to G}
	\end{align}
	
	Notice that in the Serre spectral sequence for $\Omega S^{2k}  \to G \to S^{4k-1}$,  by Proposition \ref{differential}, we have the
	identity 
	\begin{align}
		d_{4k-1}(z_{4k-2}) = H(f_a(\tilde{\pi})) · \bar{\iota}_{4k-1}, \label{equ1 Prop2}
	\end{align}
	
	Note that Pontrjagin algebra $H_\ast(\Omega S^{2k})$ is isomorphic to the tensor algebra over $H_\ast( S^{2k-1})$. $\Omega [n]=\Omega \Sigma [n]$ where last $[n]$ is the degree $n$ map between $S^{2k-1}$.
	$(\Omega [n])_\ast$ is an algebra homomorphism of  Pontrjagin algebras and $z_{4k-2}$  is  dual to the tensor square of the generator of $H_{2k-1}( S^{2k-1})$, these facts together imply that 
	\begin{align}
		(\Omega [n])^\ast (z_{4k-2}) = n^2z_{4k-2}. \nonumber
	\end{align}
	From the naturality of the Serre spectral sequences for the  diagram (\ref{diam: fibr F to G}), we have,
	\begin{align*}
		&d_{4k-1}((\Omega [n])^\ast (z_{4k-2}))=p_X^{\ast} d_{4k-1}(z_{4k-2}),  ~~~\text{i.e.} \\
		&~~~ nd_{4k-1}(nz_{4k-2})= H(f_a(\tilde{\pi})) · y_{4k-1}.
	\end{align*}
	Thus we get  $H(f_a(\tilde{\pi})) = \lambda n$ and $d_{4k-1}(nz_{4k-2})=\lambda y_{4k-1}$. Now by the derivation-property of differential, we get $d_{4k-1}(nz_{(2k-1)j}) =\lambda  y_{4k-1} \otimes  z_{(2k-1)(j-2)}$ for any $j\geq 2$.
	
\end{proof}

From Propositions \ref{prop1} and \ref{prop2}, we deduce the reduced cohomology of $F_{\tilde{\pi}}$: 
\begin{align*}
	\bar{H}^{i}(F_{\tilde{\pi}})=\left\{
	\begin{array}{ll}
		\Z\{y_{2k-1}\}, & \hbox{$i=(2k-1)$;} \\
		\Z_\lambda\{y_{1+(2k-1)j}\}, & \hbox{$i = 1 + (2k-1)j, j = 2,3, \cdots$;} \\
		0, & \hbox{otherwise.}
	\end{array}
	\right.
\end{align*}
Using the universal coefficient theorem, we get
\begin{lemma} \label{Lem:homology Fpi}
	If  $\tilde{\pi}:X\rightarrow S^{2k}$   is an  extension  of $ap_{2k}$ with $(a,n)=1$, then the reduced  homology of  $F_{\tilde{\pi}}$ are given by 
	\begin{align*}
		\bar{H}_{i}(F_{\tilde{\pi}})=\left\{
		\begin{array}{ll}
			\Z\{\hat y_{2k-1}\}, & \hbox{$i=(2k-1)$;} \\
			\Z_\lambda\{\hat y_{(2k-1)j}\}, & \hbox{$i=(2k-1)j, j=2,3,\cdots$;} \\
			0, & \hbox{otherwise.}
		\end{array}
		\right.
	\end{align*}
	where $\hat y_{2k-1}$ and $y_{2k-1}$ are dual to each other.
\end{lemma}

\section{Homomorphism $\Lambda$ and proof of Theorem \ref{thm: X fibration}}
\label{sec: Homomorphism Lambda }

\begin{lemma}\label{lem:change fibration}
	Let $X$ be a  CW complex given by (\ref{cofiber X}) and	  $S^{2k-1}\rightarrow X\xrightarrow{\pi}S^{2k}$ is a fibration with $\pi\fhe i_X=ap_{2k}$, $(a, n)=1$. Then there is a $g\in Aut(\Pn)$ such that  $X\simeq X'=\Pn\cup_{g\fhe f}e^{4k-1}$  and  there is a fibration sequence $S^{2k-1}\rightarrow X'\xrightarrow{\pi'}S^{2k}$ with $\pi'\fhe i_{X'}=p_{2k}$. 
\end{lemma}

\begin{proof}
	we have the following homotopy commutative diagram 
	\begin{align}
		\small{\xymatrix{
				S^{4k-2} \ar@{=}[d]\ar[r]^-{f}&	\Pn  \ar[d]^-{g=a\iota_P} \ar[r]^-{i_X}&X\ar[rr]^-{\pi}\ar[d]^{h}&&S^{2k}\ar@{=}[d]\\
				S^{4k-2} \ar[r]^-{g\fhe f}&	P^{2k}(n)\ar[r]^-{i_{X'}}& X'\ar[rr]^{\pi'=\pi\fhe h^{-1}} && S^{2k}
		} }. 
	\end{align}
	where  $h$, as a map between two cofibers,  is induced by the left commutative diagram.
	Clearly,  $h$ is a homotopy equivalence. The homotopy fibre $F_{\pi'}$ of $\pi'$ is also homotopy equivalent to $F_{\pi}\simeq S^{2k-1}$ and $\pi'\fhe i_{X'}= \pi\fhe h^{-1}\fhe i_{X'}= \pi\fhe i_X\fhe g^{-1}=(ap_{2k})\fhe (a^{-1}\iota_P)=p_{2k}$. 
	
\end{proof}

From  Lemma \ref{Lem:Hopf inv relation}, Remark\ref{rem: diff f(pi)} and Proposition \ref{prop2}, we have a well-defined map:
\begin{align}
	\Lambda: K_{k}^n\rightarrow \Z_{n_2},  ~~f\mapsto \bar\lambda , \label{def:bar H}
\end{align}
where $\lambda$ is the factor of  $H(f(\pi))=\lambda n$ for  an extension $\pi$~of $p_{2k}$ via $i_X$ for $X=\Pn\cup_{f}e^{4k-1}$ and $\bar \lambda$ is the modulo-$n_2$ reduction of the integral $\lambda$. 

For $g$ a self map of $P^{k+1}(n)$, if $g|_{S^k} = [t]: S^k\rightarrow S^k$ with integer $0\leq t<n$, then
we say the degree of $g$ is $t$.
 \begin{lemma}\label{lem: H(gf)}
	If $ X= P^{2k}(n) \cup_{f}e^{4k-1}$ with $p_{2k}\fhe f= 0$, then for any $g\in Aut(P^{2k}(n))$ with degree $t$,  we have 
	\begin{align*}
		\Lambda(g\fhe f)=	t^2\Lambda( f)
	\end{align*}
\end{lemma}
\begin{proof}
	Let  $\pi$ is an extension  of $p_{2k}$ via $i_{X}$ and  $X_1=\Pn\cup_{g\fhe f}e^{4k-1}$

	 $h:X\rightarrow X_1$ be the map induced by $g$. Clearly, $h$ is a homotopy equivalence since $g$ is a homotopy equivalence. Let 
	$\pi_1=[t]\fhe \pi\fhe h^{-1}$. Then 
	\begin{align*}
		\pi_1\fhe i_{X_1}=[t]\fhe \pi\fhe h^{-1}\fhe i_{X_1}= [t]\fhe \pi\fhe i_X\fhe g^{-1}= [t]\fhe p_{2k}\fhe g^{-1}= p_{2k},
	\end{align*}
	i.e., $\pi_1$ is an extension of $p_{2k}$ via $i_{X_1}$. We get the  following homotopy commutative diagram 
	\[
	\xymatrix{
		& &&S^{2k}\ar[ddd]_{[t]}\ar[r]^-{[n]} & S^{2k} \\
		S^{4k-2} \ar[r]^f \ar@{=}[d] & P^{2k}(n) \ar[r]^-{i_X}\ar[rru]^{p_{2k}} \ar[d]_{g\simeq} & X \ar[ru]_{\pi}\ar[rr]^-{~~~~~~~~p_X} \ar[d]_{h \simeq} && S^{4k-1} \ar@{=}[d]\ar[u]^{f(\pi)} \\
		S^{4k-2} \ar[r]^{g\fhe f} & P^{2k}(n)\ar[rrd]_{p_{2k}}  \ar[r]^-{i_{X_1}}  & X_1\ar@{-->}[rd]^{\pi_1}\ar[rr]^-{~~~~~~~~p_{X_1}}  && S^{4k-1} \ar[d]^{(g\fhe f)(\pi_1)} \\
		& &  &S^{2k} \ar[r]^-{[n]} & S^{2k}
	}
	\]
	
	\begin{align*}
		&((g\fhe f)(\pi_1))\fhe p_{X_1}\fhe h= [n]\fhe \pi_1\fhe h= [n]\fhe [t]\fhe \pi =[t]\fhe [n]\fhe \pi= [t]\fhe f(\pi)\fhe p_{X}. \\
		\Rightarrow ~& ((g\fhe f)(\pi_1))\fhe p_{X}= [t]\fhe f(\pi)\fhe p_{X}~~\text{by}~p_{X_1}\fhe h= p_X. 
	\end{align*}
	By the exact sequence (\ref{exact: i_X}), $(g\fhe f)(\pi_1)=[t]\fhe f(\pi)+\nu\fhe \Sigma f$ with some suspended element $\nu\in [P^{2k+1}(n), S^{2k}]$. Thus 
	\begin{align*}
		H((g\fhe f)(\pi_1))=H([t]\fhe f(\pi))=t^2 H( f(\pi)),
	\end{align*}
	which implies that $\Lambda(g\fhe f)=	t^2\Lambda( f)$.
\end{proof}

\begin{theorem}\label{thm: condition on Hf}
	Let $X$ be a  CW complex given by (\ref{cofiber X}). Then $X$
	is homotopy equivalent to an $S^{2k-1}$-fibration over $S^{2k}$   if and only if  
	\begin{itemize}
		\item $2|n$ if $k\neq 2,4$
		\item  $p_{2k}\fhe f=0$
		\item  	$\Lambda(f)=\pm \tau^2,  \tau\in \Z^{\ast}_{n_2}$
	\end{itemize}
	where $\Z^{\ast}_{n_2}$ denotes the group of all invertible elements (in the sense of multiplicative operation) in $\Z_{n_2}$.
\end{theorem}
\begin{remark}\label{rem:pi_{4k-1}(S^{2k})}
For $k>1$, it is well known that 
$[\iota_{2k-1},\iota_{2k-1}] \neq 0$ if and only if $k\neq 2,4$. By Lemma 4.1 of \cite{Toda}, 
\begin{align}
	&	\pi_{4k-1}(S^{2k})=\Z\{\alpha_{0}\}\oplus T,  ~~~T~\text{is the torsion subgroup,} \label{equ:pi4k-1(S2k)}\\
	\text{where} ~~~~	& H(\alpha_0)=1~~\text{for}~k=2,4  ~~~\text{and}~~~ H(\alpha_0)=\pm 2~\text{for}~k\neq 2,4.  \label{equ: H2(alpha0)}
\end{align}
In fact $\alpha_0$ is denoted by $\nu_4$ and $\sigma_8$ for $k=2$ and $4$ respectively in \cite{Toda}.
\end{remark}

\begin{proof}[Proof of Theorem \ref{thm: condition on Hf}]
	``$\Rightarrow$"
	\qquad
	
	By Lemma \ref{lem:change fibration},   there is a fibration sequence  $S^{2k-1}\rightarrow X'\xrightarrow{\pi'} S^{2k}$, such that $X'=\Pn\cup_{f'} e^{4k-1}$ where $f'=g\fhe f$ for some $g\in Aut(\Pn)$ and $\pi'\fhe i_{X'}=p_{2k}$.

	By Proposition \ref{prop2} and Lemma \ref{Lem:homology Fpi},   $ H(f'(\pi'))=\pm n$. From  Lemma \ref{lem: H(gf)},
	\begin{align}
	\pm 1=\Lambda(f')=\Lambda(g\fhe f)=t^2 \Lambda(f), ~\text{with}~t=deg(g).  \label{equ:H(f(pi))}
	\end{align} 
	Suppose $k\neq 2,4$, $n$ is odd.  By Remark \ref{rem:pi_{4k-1}(S^{2k})}, $H(f(\pi))=\lambda n$ with $\lambda $ even and $\Lambda(f)=\bar \lambda$, where $\pi$ is an extenison of $p_{2k}$ via $i_{X}$. This  contradicts to (\ref{equ:H(f(pi))}), i.e., for $k\neq 2,4$, $n$ must be even.
	
	Note that $(t,n_2)=1$. Let $\tau$ be the integer such that $\tau t\equiv 1$ (mod~$n_2$). Then $ \Lambda(f)=\pm \tau^2$ with $\tau\in \Z^{\ast}_{n_2}$. 
	So the proof of ``$\Rightarrow$" is finished.

	\qquad

	``$\Leftarrow$" 
	 By $\Lambda(f)=\pm \tau^2,  \tau\in \Z^{\ast}_{n_2}$, there is an extension 
	$\pi$ of $p_{2k}$ via $i_{X}$, such that 
	 \begin{align*}
H(f(\pi))=\lambda n~\text{with}~ \lambda \equiv \pm \tau^2~(\text{mod}~n_2),  (\tau, n_2)=1.
	 \end{align*}

	Let $\tau_2$ be an integer such
	that $\tau_2\tau\equiv 1 (~\text{mod}~  n_2)$. Define $\tilde \pi : X\to S^{2k}$ as the composite $\tilde \pi = [\tau_2] \fhe \pi$.
	
	We now have a commutative diagram:
	
	$$\small{\xymatrix{
			X \ar[d]^{\tilde\pi}  \ar[r]^-{q_X} &S^{4k-1} \ar[d]^{f_{\tau_2}(\tilde \pi)=[\tau_2]\fhe f(\pi)}\\
			S^{2k} \ar[r]^{[n]}& S^{2k} ~~. } }$$
	
	\begin{align*}
		H(f_{\tau_2}(\tilde \pi))=\tau_2^2H(f(\pi))=\pm n + n_2nt  ~~~\text{for some integer}~t.
	\end{align*}
	Take $\alpha\in \pi_{4k-1}(S^{2k})$ such that $H(\alpha)=-t$ or $-2t$ according as  $k=2,4$ or $k\neq 2,4$. Let $F_{\alpha\tilde{\pi}}$ be the  homotopy fibre of  $X\xrightarrow{\alpha\tilde{\pi}} S^{2k}$. From Lemma \ref{Lem:Hopf inv relation}, 
	\begin{align*}
		H(f_{\tau_2}(\alpha\tilde{\pi}))=n^2H(\alpha)+ H(f_{\tau_2}(\tilde{\pi}))=-n_2nt \pm n+n_2nt=\pm n. 
	\end{align*}
	It follows from Lemma \ref{Lem:homology Fpi} that $ F_{\alpha\tilde{\pi}}\simeq  S^{2k-1}$, since it is a simply connected homology $(2k-1)$-sphere. This implies that $X$ is the total space of an $S^{2k-1}$ fibration
	over $S^{2k}$ as required.
	
	\qquad
	
\end{proof}

In the following we give a lemma for Euler class  of vector bundles over $S^{2k}$.
\begin{lemma}\label{lem:Euler class}
	Let $n$ be any integer,  then any $n$ multiple (resp. $2n$ multiple) of a generator of $H^{2k}(S^{2k})$ can be realized as the Euler class of a $2k$ rank vector bundle over $S^{2k}$ for $k=2,4$ (resp. $k\neq 2,4$). 
\end{lemma}
\begin{proof}
	Let $\xi_{k}: \mathbb{R}^{2k}\rightarrow X\rightarrow S^{2k}$ be the rank $2k$ real vector bundle over $S^{2k}$ which is the Hopf line bundle  for $k=2,4$ and the unit tangent bundle for $k\neq 2,4$.  The Euler class $e(\xi_{k})=\bar\iota_{2k}$ or $2\bar\iota_{2k}$ according as $k=2,4$ or not, where $\bar \iota_{2k}\in H^{2k}(S^{2k})$ is a generator. For the integer $n$, the pullback  bundle $[n]^{\ast}\xi_{k}$ along the map $[n]:S^{2k}\rightarrow S^{2k}$ of degree $n$ satisfies the condition of this lemma.
\end{proof}

\begin{lemma}\label{lem:homo bar H}
 The map	$\Lambda$ in (\ref{def:bar H}) is a homomorphism. Moreover, $\Lambda$ is an epimorphism for $k=2,4$ or $k\neq 2,4$ and $2|n$. 
\end{lemma}
\begin{proof}
	Let $f_t\in K_k^n$, $t=1,2,3$, such that $\tilde f=f_1+f_2+f_3=0$. Let $\bar f=(f_1, f_2, f_3): 	\bigvee_{t=1}^3 S^{4k-2}\rightarrow \Pn$ be the map such that $f\fhe j_t=f_t$ for $t=1,2,3$ where 	$j_t:X_t\hookrightarrow \bigvee_{t=1}^{s} X_t$ is the  canonical inclusion of the $t$-th wedge summand. We have the following homotopy commutative diagram
	\begin{align}
		\xymatrix{
			& &  \Pn\ar[r]^{p_{2k}} &S^{2k}\ar[r]^-{[n]}  &  S^{2k}\\
			\bigvee_{t=1}^3 S^{4k-2} \ar[rr]^-{\bar f=(f_1,f_2 ,f_3)}  & &  \Pn\ar[r]\ar@{=}[u]  &X_{\vee}\ar[u]^{\pi}\ar[r]^-{p_{X_{\vee}}}  & 	\bigvee_{t=1}^3 S^{4k-1}\ar[u]^{\bar f(\pi)}\\
			S^{4k-2}\ar[u]^{\Delta} \ar[rr]^-{\tilde{f}=f_1+f_2+f_3}  & & \Pn\ar@{=}[u] \ar[r]^{i_{X_{+}}} & X_{+}\ar[u]^{\Delta_X}\ar[r]^-{p_{X_{+}}}  & 	S^{4k-1}\ar[u]^{\Delta} \ar@/_3pc/[uu]_{\bar f_{1}(\pi)+\bar f_{2}(\pi)+\bar f_{3}(\pi)}
		} \label{diam: VS^{4k-2}}
	\end{align}
	where the middle and bottom rows are cofibration sequences, $\pi$ is an extension of $p_{2k}$ which exists since $p_{2k}\fhe \bar f\simeq \ast$. Denote $\bar f(\pi)\fhe j_t$ by $\bar f_{t}(\pi)$, then
	\begin{align*}
		\bar f(\pi)=(\bar f_{1}(\pi), \bar f_{2}(\pi),\bar f_{3}(\pi))~\text{and}~\bar f(\pi)\fhe \Delta=\bar f_{1}(\pi)+\bar f_{2}(\pi)+\bar f_{3}(\pi).
	\end{align*}
	Clearly, $\pi\fhe \Delta_X$ is an extension of $p_{2k}$ via $i_{X_{+}}$ and   $\tilde{f}(\pi\fhe \Delta_X)$ can chosen as $\bar f_{1}(\pi)+\bar f_{2}(\pi)+\bar f_{3}(\pi)$. Since  $\tilde{f}=f_1+f_2+f_3=0$, by the following homotopy commutative diagram
	\begin{align}
		\xymatrix{
			& &  \Pn\ar[r]^{p_{2k}} &S^{2k}\ar[r]^-{[n]}  &  S^{2k}\\
			S^{4k-2} \ar[rr]^-{\tilde{f}=0}  & &  \Pn\ar[r]^-{i_{X_{+}}\simeq j_1}\ar@{=}[u]  &X_{+}=\Pn\vee S^{4k-1}\ar[u]^{p_{2k}\fhe q_1}\ar[r]^-{p_{X_{+}}}  & S^{4k-1}\ar[u]^{0}
		}
	\end{align}
	$p_{2k}\fhe q_1$ is also an extension of $p_{2k}$ via $i_{X_{+}}$ and $\tilde{f}(p_{2k}\fhe q_1)$ can be choosen as $0$, where $q_1$ is the  canonical projection to the first wedge summand. Hence 
	\begin{align*}
		&H(\tilde{f}(\pi\fhe \Delta_X))\equiv H(\tilde{f}(p_{2k}\fhe q_1))=H(0)=0~ \text{(mod $nn_2$), i.e., }\\
		&H(\bar f_{1}(\pi))+	H(\bar f_{2}(\pi))+	H(\bar f_{3}(\pi)) =	H(\bar f_{1}(\pi)+\bar f_{2}(\pi)+\bar f_{3}(\pi))\equiv 0 ~\text{(mod $nn_2$)}. 
	\end{align*}
	By the following homotopy commutative diagrams for $t=1,2,3$
	\begin{align}
		\xymatrix{
			& &  \Pn\ar[r]^{p_{2k}} &S^{2k}\ar[r]^-{[n]}  &  S^{2k}\\
			\bigvee_{t=1}^3 S^{4k-2} \ar[rr]^-{\bar f=(f_1,f_2 ,f_3)}  & &  \Pn\ar[r]\ar@{=}[u]  &X_{\vee}\ar[u]^{\pi}\ar[r]^-{p_{X'}}  & 	\bigvee_{t=1}^3 S^{4k-1}\ar[u]^{\bar f(\pi)}\\
			S^{4k-2}\ar[u]^{j_t} \ar[rr]^-{f_t}  & & \Pn\ar@{=}[u] \ar[r]^-{i_{X_t}} & X_{t}\ar[u]^{\tilde j_{t}}\ar[r]^-{p_{X'}}  & 	S^{4k-1}\ar[u]^{j_t} \ar@/_3pc/[uu]_{\bar f_{t}(\pi)}
		} 
	\end{align}
	where the middle and bottom rows are cofibration sequences, we get $\pi_{t}:=\pi\fhe \tilde{j}_t$ is an extension of $p_{2k}$ via $i_{X_t}$, where 
	$X_{t}=\Pn\cup_{f_t}e^{4k-1}$ and $f_t(\pi_t)=\bar f_{t}(\pi)$. Thus 
	\begin{align*}
		H(f_1(\pi_1))+	H(f_2(\pi_2))+	H(f_3(\pi_3)) \equiv 0 ~\text{(mod $nn_2$)}.
	\end{align*}
We  get $ \Lambda(f_1)+\Lambda(f_2)+\Lambda(f_3)=0$, which implies that $\Lambda$ is a homomorphism. 

By Lemma \ref{lem:Euler class}, if $k=2,4$ or $k\neq 2,4$ and $2|n$, then there is an $S^{2k-1}$-sphere bundle over $S^{2k}$ with Euler number $n$. From Theorem \ref{thm: condition on Hf}, its total space  is $X=\Pn\cup_{f_0}e^{4k-1}$ for some $f_0\in K_{k}^n$ with  $\Lambda(f_0)=\pm \tau^2$ and $\tau\in \Z^{\ast}_{n_2}$. Hence $\Lambda$ is an epimorphism.
\end{proof}

\begin{lemma}\label{lem:H(iv(pi))}
	$\Lambda (i_{2k-1}\fhe \nu)=0$ for any $\nu\in \pi_{4k-2}(S^{2k-1})$.
\end{lemma}
\begin{proof}
	We have the following commutative diagram of cofibration sequences
	\begin{align*}
		\small{\xymatrix{
				S^{4k-2}\ar[d]^{\nu}\ar[r]^-{i_{2k-1}\fhe \nu}&	P^{2k}(n)\ar@{=}[d] \ar[r]^-{i_{X}}&X\ar[r]^-{p_{X_1}}\ar[d]^{\pi}&S^{4k-1}\ar[d]^{f(\pi)=\Sigma\nu}\\
				S^{2k-1}\ar[r]^-{i_{2k-1}}	&	P^{2k}(n) \ar[r]^-{p_{2k}}& S^{2k}\ar[r]^{[n]} & S^{2k}	}}
	\end{align*}
	Hence	this lemma is easily obtained by $H(f(\pi))=H(\Sigma\nu)=0$. 
\end{proof}

 \begin{proof}[Proof of Theorem \ref{thm: X fibration}]
	~~By Lemmas \ref{lem: j(K)}, \ref{lem:homo bar H} and \ref{lem:H(iv(pi))}, we get the following isomorphism from the exact sequence (\ref{exact: K})
	\begin{align*}
\bar \Lambda:	j_{\ast}(K_{k}^{n})\cong \frac{K_{k}^{n}}{i_{2k-1\ast}\pi_{4k-2}(S^{2k-1})}\xrightarrow{\cong } \Z_{n_2}. 
	\end{align*}
	 If $k=2,4$ or $k\neq 2,4$ and $2|n$, then there is an $S^{2k-1}$-sphere bundle over $S^{2k}$ by Lemma \ref{lem:Euler class}. By  \cite[(5.1)]{JamesII}, its total space has homotopy type $X=\Pn\cup_{f}e^{4k-1}$ with $j_{\ast}(f)=[X_{2k}, i_{2k-1}]_r$. Hence from Theorem \ref{thm: condition on Hf},  $\bar \Lambda([X_{2k}, i_{2k-1}]_r)=\pm\tau_0^2$ with $\tau_0\in\Z_{n_2}^{\ast}$. 
	 
	 The isomorphism $\bar \Lambda$ implies the  equivalence of Theorem \ref{thm: X fibration} and Theorem \ref{thm: condition on Hf}. 
\end{proof}

\begin{proof}[Proof of Corollary \ref{corollary: Gener Sasao}]
	By Theorem \ref{thm: X fibration}, $2|n$ if $k\neq 2,4$ and	$j_{\ast}(f)=\pm\tau^{2}[X_{2k},\iota_{2k-1}]_r$, with $(\tau, n)=1$ which implies  $(\tau, n_2)=1$. Let $\tau_1$ be the integer such that $\tau_1\tau\equiv 1$ (mod $n_2$).   
	Then $g=\tau_1\iota_{P}\in Aut(\Pn)$. Let $\bar g: (\Pn, S^{2k-1})\rightarrow (\Pn, S^{2k-1})$ be the map induced by $g$. 
	\begin{align*}
		j_{\ast}(g\fhe f)=	&\bar g_{\ast}j_{\ast}(f)=\pm\tau^{2} \bar g_{\ast}[X_{2k},\iota_{2k-1}]_r=\pm\tau^{2}[\bar g_{\ast}X_{2k}, g_{\ast} \iota_{2k-1}]_r	\\
		=&\pm\tau^{2}\tau_1^2[X_{2k},\iota_{2k-1}]_r=\pm [X_{2k},\iota_{2k-1}]_r. 
	\end{align*}
	Take $\tilde{f}=g\fhe f$ or $-(g\fhe f)$ such that $j_{\ast}(\tilde{f})= [X_{2k},\iota_{2k-1}]_r$.  Then by Proposition \ref{Prop: hty equi}, 
	\begin{align*}
		X\simeq \Pn\cup_{\tilde{f}} e^{4k-1}. 
	\end{align*}
\end{proof}

\section{Proof of Theorem \ref{thm: x1=X2}}
\label{sec:proof of thm X1=X2 }

 In the last section,  we complete the proof of Theorem \ref{thm: x1=X2}. 
 
 Let $\pi_i$ and $\mu_i:=f_i(\pi_i)$ be the extensions defined in diagram (\ref{digm:fa[pi]}) with $a=1$ for  $X_i$ ($i=1,2$) in Theorem \ref{thm: x1=X2}. Clearly 
 \begin{align*}
 	i_{2k}\fhe \mu_i=\Sigma f_i ~~ (i=1,2).
 \end{align*}
 \begin{lemma}\label{lem: X1=X2 on mu}
 	Let $X_1$ and $X_2$ be given in Theorem \ref{thm: x1=X2} with $k\neq 2,4$ and $2|n$. Then $X_1\simeq X_2$ if and only if there is an integer  $t$ such that 
 	\begin{align*}
 		\text{(i)}~[t]\fhe \mu_1\equiv \pm \mu_2~\text{(mod $Ker(i_{2k\ast})$)};~~~~~\text{(ii)}~t^2\equiv \pm 1 ~\text{(mod $2n$)}.
 	\end{align*}
 \end{lemma}
 \begin{proof}	``$\Rightarrow$"
 	\qquad
 	
 	Since $X_1\simeq X_2$, from Proposition \ref{Prop: hty equi}, there is an $g\in Aut(\Pn)$ such that $g\fhe f_1=\pm f_2$. Let the degree of $g$ is integer $t$, that is $g\fhe i_{2k-1}=ti_{2k-1}$. Clearly, $(t,n)=1$. 
 	
 	For $i=1,2$, by the definition of $X_i$, we get $j_{\ast}(f_i)=[X_{2k}, \iota_{2k-1}]_r$.
 	\begin{align*}
 		&\Sigma(g\fhe f_1)=\pm \Sigma f_2 \Rightarrow (\Sigma g)\fhe i_{2k}\fhe \mu_1=\pm i_{2k}\fhe \mu_2 \Rightarrow (t i_{2k})\fhe \mu_1= \pm i_{2k}\fhe \mu_2\\
 		\Rightarrow ~& i_{2k\ast}([t]\fhe \mu_1)=\pm \mu_2\Rightarrow i_{2k\ast}([t]\fhe \mu_1\pm \mu_2)=0 \Rightarrow [t]\fhe \mu_1\equiv \pm \mu_2~\text{(mod $Ker(i_{2k\ast})$)}.
 	\end{align*}
 	Let $\bar g$ be the self-map of the CW-pair $(\Pn, S^{2k-1})$ which is induced by $g$. Then $\bar g_{\ast}([X_{2k}, \iota_{2k-1}]_r)=t^2[X_{2k}, \iota_{2k-1}]_r$, which implies 
 	\begin{align*}
 		\bar{g}_{\ast}(j_{\ast}(f_1))=j_{\ast}g_{\ast}(f_1)=j_{\ast}(\pm f_2), ~\text{i.e.,}~t^2[X_{2k}, \iota_{2k-1}]_r=\pm[X_{2k}, \iota_{2k-1}]_r. 
 	\end{align*}
 	Hence, $t^2\equiv \pm 1$ (mod $2n$). 
 	
 	\qquad
 	``$\Leftarrow$"
 	\qquad
 	Conversely, let $t$ be the integer satisfying conditions $(i)$ and $(ii)$ in Lemma \ref{lem: X1=X2 on mu}. Let $g=t\iota_P+\epsilon i_{2k-1}\eta_{2k-1}p_{2k}\in Aut(\Pn)$, $\epsilon\in\{0,1\}$. Composing  $i_{2k}$ on the left side of both sides of equation $(i)$ in Lemma \ref{lem: X1=X2 on mu}, we get $t\Sigma f_1=\pm \Sigma f_2$ which implies 
 	\begin{align}
 		\Sigma(g\fhe f_1)=\Sigma (t\iota_P)\fhe f_1 =\pm \Sigma f_2 \text{~since~}  p_{2k} f_1=0.\label{equ:Egf_1=Ef_2}
 	\end{align}
 	Moreover, by $(i)$ in Lemma \ref{lem: X1=X2 on mu},
 	\begin{align*}
 		&j_{\ast}(g\fhe f_1)=\bar g_{\ast}j_{\ast}(f_1)=\bar g_{\ast}[X_{2k}, \iota_{2k-1}]_r=\pm t^2[X_{2k}, \iota_{2k-1}]_r=\pm [X_{2k}, \iota_{2k-1}]_r.
 	\end{align*}
 	By $j_{\ast}(f_2)=[X_{2k}, \iota_{2k-1}]_r$, we get $g\fhe f_1\pm f_2\in Ker(j_{\ast})=Im(i_{2k-1\ast})$, i.e, 
 	\begin{align*}
 		g\fhe f_1\pm f_2=i_{2k-1}\fhe \xi~\text{for some}~\xi\in \pi_{4k-2}(S^{2k-1}).
 	\end{align*}
 	By (\ref{equ:Egf_1=Ef_2}), $\Sigma(i_{2k-1}\fhe \xi)=0$. By Corollary 4.3 of  \cite{James On sph bundle}, there is an $\xi'\in\pi_{4k-2}(S^{2k-1})$ such that $i_{2k-1}\fhe \xi=i_{2k-1}\fhe \xi'$ and $\Sigma \xi'=0$. By  \cite[Corollary (5.3)]{Freudenthal Thm}
 	\begin{align*}
 		\xi'=\epsilon[\eta_{2k-1},\iota_{2k-1}]~\text{for some integer}~ \epsilon\in\{0,1\}. 
 	\end{align*}
 	From Lemma \ref{lem:exist g'}, there exists a $g'\in Aut(\Pn)$ such that 
 	Let $g'=g+\xi'$. We have 
 	$g'\fhe f_1=\pm f_2$. Hence  $X_1\simeq X_2$.  
 \end{proof}
 
 \begin{lemma}\label{lem: X1=X2 on rEmu=Emu}
 	Let $X_1$ and $X_2$ be given in Theorem \ref{thm: x1=X2} with $k\neq 2,4$ and $2|n$. Then $X_1\simeq X_2$ if and only if there is an integer  $t$ such that 
 	\begin{align*}
 		~t^2\equiv 1 ~\text{(mod $2n$) and ~} t\Sigma \mu_1\equiv\Sigma \mu_2 ~\text{(mod $n\pi_{4k}(S^{2k+1})$)}
 	\end{align*}
 \end{lemma}
 \begin{proof}	``$\Rightarrow$"
 	\qquad
 	
 	If  $X_1\simeq X_2$, then $(i)$ and $(ii)$ in Lemma \ref{lem: X1=X2 on mu} are satisfied. Since $n$ is even,  $t^2\equiv -1$ (mod $2n$) has no solution. By $(ii)$ in Lemma \ref{lem: X1=X2 on mu}, $t^2\equiv 1$ (mod $2n$). Therefore, $(i)$ in Lemma \ref{lem: X1=X2 on mu} becomes $[t]\fhe \mu_1\equiv  \mu_2~\text{(mod $Ker(i_{2k\ast})$)}$. By suspension, it follows \cite[Theorem 4.2]{James On sph bundle} that
 	\begin{align*}
 		t\Sigma\mu_1\equiv \Sigma\mu_2 ~\text{(mod $n\pi_{4k}(S^{2k+1})$)}. 
 	\end{align*}
 	
 	``$\Leftarrow$"
 	
 	Conversely, suppose that there is an integer $t$ such that $t^2\equiv 1$ (mod $2n$) and $r\Sigma \mu_1\equiv\Sigma \mu_2$  (mod $n\pi_{4k}(S^{2k+1})$). Clearly, $(ii)$ in Lemma \ref{lem: X1=X2 on mu} is satisfied. 
 	
 	By the EHP sequence (exact)
 	\begin{align}
 		\pi_{4k-2}(S^{2k-1})\xrightarrow{\Sigma} \pi_{4k-1}(S^{2k})\xrightarrow{H_2} \pi_{4k-1}(S^{4k-1})\cong \Z, \nonumber
 	\end{align}
 	the restriction of $H_2$ on the torsion subgroup $T$ is trivial, where $T$ comes from Remark \ref{rem:pi_{4k-1}(S^{2k})}. 
 	
 	The suspension $\Sigma: \pi_{4k-1}(S^{2k})\rightarrow  \pi_{4k}(S^{2k+1})$ is epimorphism, and its kernel is $\Z\{\omega_{2k}\}$, where $\omega_{2k}:=[\iota_{2k},\iota_{2k}]$ \cite[(2.11) and page 35]{Toda}  and $H_{2}(\omega_{2k})=\pm 2\iota_{4k-1}$ \cite[Proposition 2.7]{Toda} . 
 	By  Proposition \ref{prop2} and Theorem \ref{thm: condition on Hf},  $H_{2}(\mu_1)=\lambda_1n\iota_{4k-1}$ for some integer $\lambda_1$ which is  coprime with $n$. Thus
 	$\mu_1-\frac{\pm\lambda_1n}{2}\omega_{2k}\in Ker(\Sigma)$. Replacing $\lambda_1$ by $-\lambda_1$ if necessary, we get
 	\begin{align*}
 		\mu_1=\frac{\lambda_1n}{2}\omega_{2k}+\nu_1\equiv \frac{n}{2}\omega_{2k}+\nu_1~\text{(mod $n\pi_{4k-1}(S^{2k})$)}  ~\text{for some}~\nu_1\in T.
 	\end{align*}
 	Similarly  there are integer $\lambda_2$ with $(\lambda_2, n)=1$ and  torsion element~$\nu_2$ such that $\mu_2\equiv \frac{n}{2}\omega_{2k}+\nu_2~\text{(mod $n\pi_{4k-1}(S^{2k})$)}$. Hence 
 	\begin{align*}
 		t \mu_1-\mu_2\equiv \frac{(t-1)n}{2}\omega_{2k}+\nu_3~\text{(mod $n\pi_{4k-1}(S^{2k})$)} ~\text{for some}~\nu_3\in T. 
 	\end{align*}
 	$t\Sigma \mu_1\equiv\Sigma \mu_2$  (mod $n\pi_{4k}(S^{2k+1})$) implies 
 	$\Sigma(t \mu_1-\mu_2)\equiv 0$ (mod $n\pi_{4k}(S^{2k+1})$). So 
 	$\Sigma\nu_3\equiv 0$ (mod $n\pi_{4k}(S^{2k+1})$). As mentioned in page 35 of \cite{Toda}, $\Sigma: \Sigma\pi_{4k-2}(S^{2k-1})\rightarrow \pi_{4k}(S^{2k+1})$ is injective.  Thus
 	\begin{align*}
 		&t \mu_1-\mu_2\equiv \frac{(t-1)n}{2}\omega_{2k}~\text{(mod $n\pi_{4k-1}(S^{2k})$)}\\ 
 		\Rightarrow~	& [t]\fhe\mu_1-\mu_2=t \mu_1-\mu_2+\frac{t(t-1)}{2}w_{2k}H_2(\mu_1)\\
 		\equiv &\frac{(t-1)n}{2}\omega_{2k}+\frac{t(t-1)}{2}\lambda_1nw_{2k}
 		~\text{(mod $n\pi_{4k-1}(S^{2k})$)}\\
 		\equiv &0 ~\text{(mod $n\pi_{4k-1}(S^{2k})$), since~$t$ is odd} .
 	\end{align*}
 	It follows \cite[Theorem 4.2]{James On sph bundle} that $(i)$ in Lemma \ref{lem: X1=X2 on mu} holds. 
 \end{proof}

 \begin{proof}[Proof of Theorem \ref{thm: x1=X2}]
 	If $t\Sigma^{\infty} f_1=\Sigma^{\infty}f_2$, with $t^2\equiv 1$ (mod $2n$), then $i_{2k\ast}(t\Sigma^{\infty} \mu_1)=i_{2k\ast}(\Sigma^{\infty} \mu_2)$. Hence
 	\begin{align*}
 		t\Sigma^{\infty} \mu_1-\Sigma^{\infty} \mu_2\in Ker(i_{2k\ast}:\pi^{s}_{4k-1}(S^{2k})\rightarrow \pi^{s}_{4k-1}(P^{2k+1}(n)))=n\pi^s_{4k-1}(S^{2k}). 
 	\end{align*} 
 	Since $\Sigma\mu_1, \Sigma\mu_2\in \pi_{4k}(S^{2k+1})$ are in stable range, we get 
 	$i_{2k+1\ast}(t\Sigma\mu_1)=i_{2k+1\ast}(\Sigma \mu_2)$, which implies  that 
 	$t\Sigma\mu_1- \Sigma\mu_2\in Ker(i_{2k+1\ast})=n\pi_{4k}(S^{2k+1})$, from Lemma \ref{lem: X1=X2 on rEmu=Emu}, $X_1\simeq X_2$. The proof of sufficiency in Theorem \ref{thm: x1=X2} is obvious.
 \end{proof}

 \qquad
 
 \qquad
 
\noindent
{\bf Acknowledgement.}
The first author was partially supported by National Natural Science Foundation of China (Grant No. 12571075); the second author was partially supported by National Natural Science Foundation of China (Grant No. 11971461 and  12571075).

\bibliographystyle{amsplain}

\end{document}